\documentclass[11pt]{amsart}

\usepackage{graphicx}
\usepackage{amsfonts,amssymb}

\usepackage{hyperref}
\hypersetup{pdftex,colorlinks=true,allcolors=black}

\usepackage{amsmath,amsthm,latexsym,mathabx,xcolor,tikz,tkz-graph,standalone}
\newtheorem{al}{Algorithm}

\def\thi{{\theta_\infty}}

\def\b4{${\mathcal B}_{4}$}

\def\g2{${\mathcal G}_{2}$}
\def\gg2{{\mathcal G}_{2}}

\def\mg2{${\mathcal M}_{{\mathcal G}_{2}}$}
\def\mmg2{{\mathcal M}_{{\mathcal G}_{2}}}
\def\hhmg2{\widehat{\mathcal{M}}_{\mathcal{G}_2}}
\def\hmg2{$\widehat{\mathcal{M}}_{\mathcal{G}_2}$}
\def\p6{PVI}
\def\h6{H_{VI}}
\def\mp6{${\mathcal M}_{PVI}$}
\def\mmp6{{\mathcal M}_{PVI}}

\DeclareMathOperator{\Tr}{Tr}
\newcommand{\be}{\begin{equation}}
\newcommand{\eneq}{\end{equation}}

\newcommand{\complessi}{\mathbb C}

\theoremstyle{plain}
\newtheorem{theorem}{Theorem}[section]
\newtheorem{lemma}[theorem]{Lemma}
\newtheorem{corollary}[theorem]{Corollary}
\newtheorem{proposition}[theorem]{Proposition}

\theoremstyle{definition}
\newtheorem{definition}{Definition}[section]

\begin{document}

\title[Finite orbits of the braid group on the Garnier system]{Finite orbits of the pure braid group on the monodromy of the $2$-variable Garnier system}

\author{P. Calligaris and M. Mazzocco}\thanks{Department of Mathematical Sciences,
Loughborough University, Leicestershire, LE11 3TU, United Kingdom; email:
m.mazzocco@lboro.ac.uk}

\maketitle

\begin{abstract}
In this paper we show that the $\text{SL}_{2}(\complessi)$ character variety of  the Riemann sphere $\Sigma_5$ with five boundary components is a $5$--parameter family of affine varieties of dimension $4$. We endow this family of affine varieties with an action of the braid group and classify exceptional finite orbits.  This action represents the nonlinear monodromy of the $2$ variable Garnier system and finite orbits correspond to its algebraic solutions. 
\end{abstract}

\section{Introduction}

The Garnier system ${\mathcal G}_2$ is the  isomonodromy deformation of the following two-dimensional
Fuchsian system:
\begin{equation}
{{\rm d}\over{\rm d}\lambda} \Phi=\left(
\sum_{k=1}^{4}{{\mathcal A}_k\over \lambda-a_k} \right)\Phi,\qquad \lambda\in\complessi,
\label{N1in}
\end{equation}
$a_1,\dots,a_{4},$ being pairwise distinct complex numbers. 
The residue matrices ${\mathcal A}_j$ satisfy the following conditions:
$$
{\rm eigen}\left({\mathcal A}_j\right)=\pm{\theta_j\over2}
\quad\hbox{and}\quad
-\sum_{k=1}^{n+2}{\mathcal A}_k={\mathcal A}_\infty,
$$
where $\theta_j\in\mathbb C$, $j=1,\dots,4$ and we assume
$$
\hbox{for}\,\theta_\infty\neq 0,\quad
{\mathcal A}_\infty:={1\over2}\left(
\begin{array}{cc}
\theta_\infty & 0\\ 0 & -\theta_\infty\\
\end{array}\right),
$$
with $\theta_\infty\in\mathbb C$.

The Riemann-Hilbert correspondence associates to each Fuchsian system (\ref{N1in}) its monodromy representation class, or in other words, a point in 
 the moduli space of rank two linear monodromy representations over the $2$-dimensional sphere $\Sigma_5$ with five boundary components:
\begin{equation*}
\mmg2 := \text{Hom}(\pi_1(\Sigma_5),\text{SL}_{2}(\complessi))/\text{SL}_{2}(\complessi),
\end{equation*}
also called {\it $\text{SL}_{2}(\complessi)$ character variety of $\Sigma_5$.}

After fixing a basis of oriented loops $\gamma_1,\dots,\gamma_4,\gamma_\infty$ for $\pi_1(\Sigma_5)$ such that $\gamma_\infty^{-1} = \gamma_1\cdots\gamma_4$, as in Figure \ref{fig:4cuts}, an equivalence class of an homomorphism in the character variety \mg2 is determined by the five matrices $M_1,\dots,M_4$, $M_\infty\in \text{SL}_{2}(\complessi)$, that are images of $\gamma_1,\dots,\gamma_4,\gamma_\infty$. These matrices must satisfy the relation:
\begin{equation}
M_\infty M_4 M_3 M_2 M_1 = \mathbb{I}.
\label{eq:cyclic}\end{equation}

\begin{figure}
\centerline{%
\includegraphics[width=9cm]{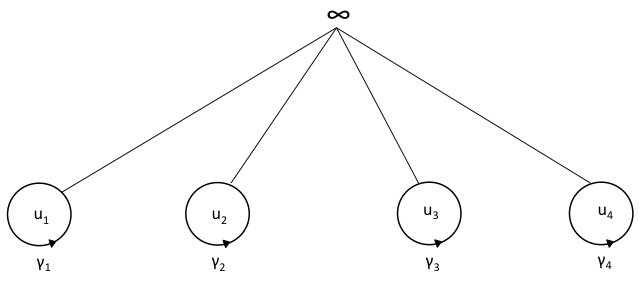}}
\caption{The basis of loops for $\pi_1(\Sigma_5)$.}
\label{fig:4cuts}
\end{figure}
 
In this paper we assume that  $M_\infty$ is diagonalizable:
$$
{\rm eigen}(M_\infty) = 
e^{\pm\pi i {\thi}}
$$
As a consequence the character variety \mg2 is identified with the quotient space \hmg2, defined as:
\begin{equation}
\widehat{\mathcal{M}}_{\mathcal{G}_2} := \left\lbrace (M_1,\dots,M_4)\in \text{SL}_{2}(\complessi) |   {\rm eigen}( M_4 M_3 M_2 M_1) = e^{\pm\pi i {\thi}}
\right\rbrace / \sim,
\label{eq:Mmg2}\end{equation}
where $\sim$ is equivalence up to simultaneous conjugation of $M_1,\dots,M_4$ by a matrix in $\text{SL}_{2}(\complessi)$. 

As the pole positions $a_1,\dots,a_4$ in \eqref{N1in} vary in the configuration space of $4$ points,  the monodromy matrices 
${M}_1,\dots,{M}_{4}$ 
of (\ref{N1in}) remain constant if and only if (see \cite{Bol}) the residue matrices 
${\mathcal A}_1,\dots,{\mathcal A}_{4}$ are solutions of the Schlesinger
equations  \cite{Sch} which in the  $2\times 2$ case reduce 
to the Garnier system ${\mathcal G}_2$  \cite{Gar1,Gar2}. The structure of analytic continuation of the solutions of the Garnier system is described by a certain action of the   pure braid group $P_4$ \cite{MazzoccoPVI} (see also \cite{Cousin2016}) that can be deduced from the following action of the braid group $B_4$:
\begin{equation}
B_4 \times \widehat{\mathcal{M}}_{\mathcal{G}_2} \longmapsto \widehat{\mathcal{M}}_{\mathcal{G}_2},
\label{eq:action}\end{equation}
defined in terms of the following generators:
\begin{equation}
\begin{array}{ll}
\sigma_1 : (M_1,M_2,M_3,M_4) &\mapsto (M_2,M_2M_1M_2^{-1},M_3,M_4),\\
\sigma_2 : (M_1,M_2,M_3,M_4) &\mapsto (M_1,M_3,M_3 M_2M_3^{-1},M_4),\\
\sigma_3 : (M_1,M_2,M_3,M_4) &\mapsto (M_1,M_2,M_4,M_4 M_3M_4^{-1}),\\
\end{array}
\label{eq:braidaction}\end{equation}
so that $M_\infty$ is preserved.

Our aim in this paper is to classify the finite orbits of this action. In our classification we exclude the case when the monodromy group $\langle M_1,\dots,M_4\rangle$ is reducible because in this case the Garnier system can be solved in terms of Lauricella hypergeometric functions \cite{mazzoccoClassical}, and the case in which one of the monodromy matrices is a root of the identity because in this case the Garnier system reduces to the sixth Painlev\'e equation  \cite{mazzoccoClassical} for which all algebraic solutions are classified in \cite{LT}. 
Therefore we restrict to the following open set:
\begin{align}\label{eq:bigopen}\mathcal U = 
\big\{ (M_1,\dots,M_4)& \in\hhmg2 |  \langle M_1,\dots,M_4\rangle \hbox{ irreducible, } \\
 & M_i\neq\pm \mathbb{I}, \forall i=1,\dots,4,\infty\big\} / \sim,\nonumber
\end{align}

To explain our classification result, we firstly identify the open set $\mathcal U$ with an affine variety:
\begin{lemma}\label{lm:u1}
Let the functions $p_i$, $p_{ij}$, $p_{ijk}$ be defined as:
\begin{equation}
\begin{array}{lllll}
&p_{i} &= \Tr M_{i}, \quad   &\quad i=1,\dots,4,\\
&p_{ij} &= \Tr M_{i}M_{j}, \quad   &\quad i,j=1,\dots,4, &\quad i > j,\\
&p_{ijk} &= \Tr M_{i}M_{j}M_{k}, \quad   &\quad i,j,k=1,\dots,4, &\quad i > j > k,\\
&p_\infty &= \Tr M_4 M_3 M_2 M_1,\\
\end{array}
\label{eq:pidef}\end{equation}
then, for every choice of $p_1,\dots,p_4,p_\infty$, the open set of monodromy matrices $\mathcal U$  is isomorphic to a four dimensional affine variety:
\begin{equation}
\mathcal A:=\complessi [p_{21},p_{31},p_{32},p_{41},p_{42},p_{43},p_{321},p_{432},p_{431},p_{421}] / I,
\label{eq:algv}\end{equation}
where $I$ is the ideal generated by the algebraically dependent polynomials $f_1,\dots,f_{15}$ defined in \eqref{eq:f321}-\eqref{eq:f11}. 
\end{lemma}
Therefore we think of $p_1,\dots,p_4,p_\infty$ as a set of parameters and of $p_{ij}$, $p_{ijk}$ as an overdetermined system of coordinates on $\mathcal U$ and we express the action \eqref{eq:action} in terms of $p_i$, $p_{ij}$, $p_{ijk}$ as follows:
\begin{lemma}\label{lm:u2}
The following maps  $\sigma_i :\mathcal A \longrightarrow\mathcal A$, $i=1,2,3$, acting on the coordinates 
\begin{equation}
p:=(p_1,p_2,p_3,p_4,p_{\infty},p_{21},p_{31},p_{32},p_{41},p_{42},p_{43},p_{321},p_{432},p_{431},p_{421}) \in\complessi^{15},
\label{15tuple}\end{equation}
 as follows:
\begin{equation}
\begin{array}{ll}
\sigma_1 : &p \mapsto  (p_2, p_1, p_3, p_4, p_\infty, p_{21}, p_{32}, 
 p_1 p_3 - p_{31} - p_{21} p_{32} + p_2 p_{321}, p_{42},\\ & 
 p_1 p_4 - p_{41} - p_{21} p_{42} + p_2 p_{421}, p_{43}, p_{321}, 
 p_1 p_{43} - p_{431} - p_{21} p_{432} + p_2 p_\infty,\\ & p_{432}, p_{421}),\\
\sigma_2 : &p \mapsto (p_1, p_3, p_2, p_4, p_\infty, p_{31}, 
 p_1 p_2 - p_{21} - p_{31} p_{32} + p_3 p_{321}, p_{32}, p_{41}, p_{43}, \\ &
 p_2 p_4 - p_{42} - p_{32} p_{43} + p_3 p_{432}, p_{321}, p_{432}, 
 p_2 p_{41} - p_{421} - p_{32} p_{431} + p_3 p_\infty,\\ & p_{431}),\\
\sigma_3 : &p \mapsto (p_1, p_2, p_4, p_3, p_\infty, p_{21}, p_{41}, p_{42}, 
 p_1 p_3 - p_{31} - p_{41} p_{43} + p_4 p_{431}, \\ &
 p_2 p_3 - p_{32} - p_{42} p_{43} + p_4 p_{432}, p_{43}, p_{421}, p_{432}, p_{431}, \\ &
 p_{21} p_3 - p_{321} - p_{421} p_{43} + p_4 p_\infty),
\end{array}\label{eq:braidpaction}
\end{equation}
define an action of the braid group $B_4$ on $\mathcal A$.
\end{lemma}

Therefore, our problem is to find all points $p\in\mathcal A$ such that their orbit under the action of the pure braid group $P_4$  induced by the action \eqref{eq:braidpaction} of the braid group $B_4$
is finite. Incidentally, the action \eqref{eq:braidpaction} can also be interpreted as the Mapping Class Group action on the the character variety \mg2.

Our approach is based on the observation that given $p\in\mathcal A$ such that it generates a finite orbit under the action of the pure braid group $P_4$, then for any subgroup $H\subset P_4$ the action of $H$ over $p\in\mathcal A$ must also produce a finite orbit. We select four subgroups $H_1,H_2,H_3,H_4 \subset P_4$  such that the restricted action is isomorphic to the action of the pure braid group $P_3$ on the $SL_2(\complessi)$ character variety of the Riemann sphere with four boundary components \mp6 that can be identified with:
\begin{align}
\widehat{\mathcal{M}}_{PVI} := \left\lbrace (N_1,N_2,N_3)\in \text{SL}_{2}(\complessi) |   \right.& N_\infty N_3 N_2 N_1 = \mathbb{I},\nonumber\\ 
& N_\infty =\exp( i \pi \theta_{\infty}\sigma_3),~\theta_{\infty}\in\complessi \left.\right\rbrace / \sim.
\label{eq:Mmp6}\end{align}

In other words, we show that in order for a point $p\in\mathcal A$ to belong to a finite orbit of the pure braid group $P_4$, it must have $4$ projections on points  $q=(q_1,q_2,q_3,q_{\infty},$ $q_{21},q_{31},q_{32})$ that have a finite orbit under the pure braid group $P_3$. 

We then invert this way of thinking: since all finite orbits of the pure braid group $P_3$ on $q=(q_1,q_2,q_3,q_{\infty},$ $q_{21},q_{31},q_{32})$ have been classified in Lisovyy and Tykhyy's work \cite{LT}, we start from  their list and reconstruct {\it candidate points} $p\in\mathcal A$ that satisfy the necessary conditions to belong to a finite orbit. We then classify all candidate points that indeed produce finite orbits. In order to avoid redundant solutions to this classification problem, we introduce the symmetry group $G$ of the affine variety \eqref{eq:algv} and factorize our classification modulo the action of $G$. The action of the symmetry group $G$ on $\mathcal A$ is calculated in the Appendix using known results about B\"acklund transformations of Schlesinger equations  \cite{MazzoccoCanonical}.
    
In order to produce our candidate points we use the classification result in \cite{LT} that shows that there are $4$ types of finite orbits of the braid group $B_3$:
\begin{enumerate}
\item Fixed points corresponding to Okamoto's Riccati solutions \cite{Ok2}.
\item {\it Kitaev orbits,} corresponding to algebraic solutions obtained by the \textit{pull-back} of the hypergeometric equation (see \cite{MR1911252} and \cite{MR1909248}).
\item  {\it Picard orbits,} corresponding to algebraic solutions obtained in terms of the Weierstrass elliptic function  (see \cite{EmilePicard1889} and \cite{MazzoccoChazy}).
\item $45$ exceptional finite orbits \cite{LT}.
\end{enumerate}

In order to keep down  the number of pages and of technical lemmata, we restrict our classification to {\it exceptional  orbits,} namely orbits for which  the corresponding monodromy group is not reducible, none of the monodromy matrices is a multiple of the identity and 
at most one projection giving either a Kitaev or a Picard orbit is allowed.  Therefore, our classification does not include the solutions found by Tsuda \cite{Tsuda2006657} by calculating fixed points of bi-rational canonical transformations, nor the ones found by 
Diarra in \cite{Diarra2013} using the method of \textit{pull-back} introduced in \cite{MR1909248} and \cite{MR1911252}, nor the families of algebraic solutions obtained by Girand in \cite{MR3499084} by restricting a logarithmic flat connection defined on the complement of a quintic curve on $\mathbb{P}^2$ on generic lines of the projective plane - indeed these algebraic solutions have at least two projections giving Kitaev orbits.
   
Our final classification result consists in a list of $54$ exceptional finite orbits of the action \eqref{eq:braidpaction} obtained up to the action of the group of symmetries $G$ (see Table \ref{tb:classification}).  Interestingly, $53$ of these orbits correspond to finite monodromy groups\footnote{We are grateful to Gael Cousin for asking us this question.} so that our result relates to the problem of classifying  the representations of the $SL_2(\mathbb C)$  character variety of the Riemann sphere with five boundary components on finite groups. One orbit (element $25$ in Table \ref{tb:classification}) corresponds to an infinite monodromy group despite the fact that all of its projections to points corresponding to PVI generate finite monodromy groups.

From the monodromy data $M_1,\dots,M_{4}$, it is in principle possible to recover the explicit formulation of the associated solution of \g2 using the method developed by Lisovyy and Gavrylenko in \cite{2016arXiv160800958G} of Fredholm determinant representation for isomonodromic tau functions of Fuchsian systems of the form \eqref{N1in}. However, the shortest finite orbit classified in our paper has length $36$, for this reason the associated algebraic solution of \g2 has eventually $36$ branches and we doubt that the expression of this solution can have a nice and compact form.

All the algorithms necessary to produce this classification can be found in \cite{ouralgorithms}.

\section{Co-adjoint coordinates on \mg2}\label{sec:coords}

As explained in the introduction, 
we identify the character variety \mg2 with the quotient space \hmg2 defined in \eqref{eq:Mmg2}. Following \cite{P1,P2}, the first step to  endow \hmg2 with a system of co-adjoint coordinates is to introduce a parameterization of the monodromy matrices in terms of their traces and traces of their products. The following result is a generalization of a result proved by Iwasaki for the case of the sixth Painlev\'e equation \cite{iwasakicmp}:

\begin{theorem}\label{lm:coord}
Let $(M_1,\dots,M_4) \in\mathcal U$, $p\in\mathcal A$ as in Lemma \ref{lm:u1} and 
$g(x,y,z) := x^2+y^2+z^2-xyz-4$, then in the open set:
\begin{equation}
\mathcal{U}^{(0)}_{jk} := \hhmg2 \cap \{ (p_{jk}^2-4)g( p_{jk}, p_{l}, p_{jkl} ) \neq 0 \},
\label{eq:openuno}\end{equation}
there exists a global conjugation $P\in SL_2(\mathbb C)$ such that the matrices $M_1,\dots,M_4$ can be parametrized as follows (up  to conjugation by $P$):
\begin{align}
 M_l = \left(
\begin{array}{cc}
 \frac{p_{jkl}-p_{l} \lambda^{-}_{jk}}{r_{jk}} & -\frac{g(p_{jk},p_{l},p_{jkl})}{r_{jk}^2} \\
 1 & -\frac{p_{jkl}-p_{l} \lambda^{+}_{jk}}{r_{jk}} \\
\end{array}
\right),\quad &
 M_k = \left(
\begin{array}{cc}
 -\frac{p_{j}-p_{k} \lambda^{+}_{jk}}{r_{jk}} & -\frac{y_{kl}-y_{jl} \lambda^{-}_{jk}}{r_{jk}^2} \\
 \frac{y_{kl}-y_{jl} \lambda^{+}_{jk}}{g(p_{jk},p_{l},p_{jkl})} & \frac{p_{j}-p_{k} \lambda^{-}_{jk}}{r_{jk}}
\\
\end{array}
\right),\nonumber \\
M_j = \left(
\begin{array}{cc}
 -\frac{p_{k}-p_{j} \lambda^{+}_{jk}}{r_{jk}} & -\frac{y_{jl}-y_{kl} \lambda^{+}_{jk}}{r_{jk}^2} \\
 \frac{y_{jl}-y_{kl} \lambda^{-}_{jk}}{g(p_{jk},p_{l},p_{jkl})} & \frac{p_{k}-p_{j} \lambda^{-}_{jk}}{r_{jk}}
\\
\end{array}
\right),\quad &
M_i = \left(
\begin{array}{cc}
 \frac{p_{ijk}-p_{i} \lambda^{-}_{jk}}{r_{jk}} & -\frac{y_{il} + y_{ijkl} \lambda^{+}_{jk}}{r_{jk}^2}
\\
 \frac{y_{il} + y_{ijkl} \lambda^{-}_{jk}}{g(p_{jk},p_{l},p_{jkl})} & -\frac{p_{ijk}-p_{i} \lambda^{+}_{jk}}{r_{jk}}
\\
\end{array}
\right). \label{eq:paramT}
\end{align}
Alternatively on the open set:
\begin{equation}
\mathcal{U}^{(1)}_{jk} := \hhmg2 \cap \{ (p_{jk}^2-4)g(p_{j},p_{k},p_{jk}) \neq 0 \},
\end{equation}
the matrices $M_1,\dots,M_4$ can be parametrized as follows (up  to conjugation by $P$):
\begin{align}
M_l= \left(
\begin{array}{cc}
 \frac{p_{jkl}-p_{l} \lambda^{-}_{jk}}{r_{jk}} & -\frac{y_{kl}-y_{jl} \lambda^{+}_{jk}}{r_{jk}^2} \\
 \frac{y_{kl}-y_{jl} \lambda^{-}_{jk}}{g(p_{jk},p_{j},p_{k})} & -\frac{p_{jkl}-p_{l} \lambda^{+}_{jk}}{r_{jk}}
\\
\end{array}
\right),\quad &
M_k  = \left(
\begin{array}{cc}
 -\frac{p_{j}-p_{k} \lambda^{+}_{jk}}{r_{jk}} & -\frac{g(p_{jk},p_{j},p_{k})}{r_{jk}^2} \\
 1 & \frac{p_{j}-p_{k} \lambda^{-}_{jk}}{r_{jk}} \\
\end{array}
\right),\nonumber \\
 M_j= \left(
\begin{array}{cc}
 -\frac{p_{k}-p_{j} \lambda^{+}_{jk}}{r_{jk}} & \frac{g(p_{jk},p_{j},p_{k}) \lambda^{+}_{jk}}{r_{jk}^2}
\\
 -\lambda^{-}_{jk} & \frac{p_{k}-p_{j} \lambda^{-}_{jk}}{r_{jk}} \\
\end{array}
\right),\quad &
 M_i = \left(
\begin{array}{cc}
 \frac{p_{ijk}-p_{i} \lambda^{-}_{jk}}{r_{jk}} & -\frac{y_{ik}-y_{ij} \lambda^{+}_{jk}}{r_{jk}^2} \\
 \frac{y_{ik} - y_{ij} \lambda^{-}_{jk}}{g(p_{jk},p_{j},p_{k})} & -\frac{p_{ijk}-p_{i} \lambda^{+}_{jk}}{r_{jk}} \\
\end{array}
\right).
\end{align}
Finally, on the open set:
\begin{equation}
\mathcal{U}^{(2)}_{jk} := \hhmg2 \cap \{ (p_{jk}^2-4)g(p_{jk},p_{i},p_{ijk}) \neq 0 \},
\end{equation}
the matrices $M_1,\dots,M_4$ can be parametrized as follows (up  to conjugation by $P$):
\begin{align}
M_l = \left(
\begin{array}{cc}
 \frac{p_{jkl}-p_{l} \lambda^{-}_{jk}}{r_{jk}} & -\frac{ y_{il} + y_{ijkl} \lambda^{-}_{jk}}{r_{jk}^2}
\\
 \frac{y_{il} + y_{ijkl}
  \lambda^{+}_{jk}}{g(p_{jk},p_{i},p_{ijk})} & -\frac{p_{jkl}-p_{l} \lambda^{+}_{jk}}{r_{jk}}
\\
\end{array}
\right),\quad
M_k = \left(
\begin{array}{cc}
 -\frac{p_{j}-p_{k} \lambda^{+}_{jk}}{r_{jk}} & -\frac{y_{ik}-y_{ij} \lambda^{-}_{jk}}{r_{jk}^2} \\
 \frac{y_{ik}-y_{ij} \lambda^{+}_{jk}}{g(p_{jk},p_{i},p_{ijk})} & \frac{p_{j}-p_{k} \lambda^{-}_{jk}}{r_{jk}}
\\
\end{array}
\right),\nonumber \\
M_j  = \left(
\begin{array}{cc}
 -\frac{p_{k}-p_{j} \lambda^{+}_{jk}}{r_{jk}} & -\frac{y_{ij}-y_{ik} \lambda^{+}_{jk}}{r_{jk}^2} \\
 \frac{y_{ij}-y_{ik} \lambda^{-}_{jk}}{g(p_{jk},p_{i},p_{ijk})} & \frac{p_{k}-p_{j} \lambda^{-}_{jk}}{r_{jk}}
\\
\end{array}
\right),\quad
M_i = \left(
\begin{array}{cc}
 \frac{p_{ijk}-p_{i} \lambda^{-}_{jk}}{r_{jk}} & -\frac{g(p_{jk},p_{i},p_{ijk})}{r_{jk}^2} \\
 1 & -\frac{p_{ijk}-p_{i} \lambda^{+}_{jk}}{r_{jk}} \\
\end{array}
\right),
\end{align}
where:
\begin{eqnarray}
&& r_{jk}:=\sqrt{p_{jk}^2 - 4},\qquad \lambda_{jk}^{\pm} = {p_{jk} \pm r_{jk}\over 2},\label{eq:lambda-r}\\
&& y_{kl}	:= 2 p_{kl} + p_{jk} p_{jl} - p_{j} p_{jkl} - p_{k} p_{l},\nonumber\\
&& y_{jl}	:= 2 p_{jl} + p_{jk} p_{kl} - p_{k} p_{jkl} - p_{j} p_{l},\label{eq:yil}\\
&& y_{ik}	:= 2 p_{ik} + p_{ij} p_{jk} - p_{j} p_{ijk} - p_{i} p_{k},\label{eq:yik}\\
&& y_{ij}	:= 2 p_{ij} + p_{ik} p_{jk} - p_{k} p_{ijk} - p_{i} p_{j},\label{eq:yij}\\
&& y_{il}	:= 2 p_{il} +p_{ijk} p_{jkl} - p_{jk} p_{ijkl} - p_{i} p_{l},\nonumber\\
&& y_{ijkl}	:= 2 p_{ijkl}-p_{il} p_{jk}-p_{i} p_{jkl}-p_{ijk}p_{l}+p_{i} p_{jk} p_{l}.\label{eq:yijkl}
\end{eqnarray}
\end{theorem}

\proof
Consider $(M_1,\dots,M_4)\in\mathcal U$. We only prove the statement for the open subset $\mathcal{U}^{(0)}_{jk}$. For the parametrizations on the open subsets $\mathcal{U}^{(1)}_{jk}$ and $\mathcal{U}^{(2)}_{jk}$ a similar proof applies. Under the hypothesis that $p_{jk}\neq \pm 2$,  there exists a matrix $P\in\text{SL}_{2}(\complessi)$ such that the product matrix $M_j M_k$ can be brought into diagonal form:
\begin{equation}\label{eq:lambdakj}
\Lambda_{jk} := P(M_j M_k)P^{-1} = \text{diag}\{\lambda_{jk}^{+},\lambda_{jk}^{-} \},
\end{equation}
where the eigenvalues $\lambda_{jk}^{\pm}$  are given in \eqref{eq:lambda-r}, where the positive branch of the square root is chosen. Consequently
we conjugate by $P$ the matrices $M_l,M_k,M_j,M_i$ as follows:
\begin{equation}
P (M_l,M_k,M_j,M_i ) P^{-1} = (U,V,W,T).
\end{equation}

Since, $W = \Lambda_{jk} V^{-1}$, we only need to produce the parametrization of the matrices $U,V,T$. Solving the equations $\Tr U = p_{l}$, $\Tr \Lambda_{jk}U = p_{jkl}$, and $\Tr V = p_{k}$, $\Tr \Lambda_{jk} V^{-1} = p_j$ and $\Tr T = p_i$ and $\Tr TWV = \Tr T \Lambda_{jk} = p_{ijk}$  we obtain
the diagonal elements of $U$, $V$ and $T$ respectively:
\begin{equation}
u_{11} = \frac{p_{jkl}-p_{l} \lambda^{-}_{jk}}{r_{jk}},\qquad
u_{22} = -\frac{p_{jkl}-p_{l} \lambda^{+}_{jk}}{r_{jk}},
\end{equation} 
\begin{equation}
v_{11} = -\frac{p_{j}-p_{k} \lambda^{+}_{jk}}{r_{jk}},\qquad
v_{22} = \frac{p_{j}-p_{k} \lambda^{-}_{jk}}{r_{jk}},
\end{equation}
\begin{equation}
t_{11} = \frac{p_{ijk}-p_{i} \lambda^{-}_{jk}}{r_{jk}}, \qquad
t_{22} = -\frac{p_{ijk}-p_{i} \lambda^{+}_{jk}}{r_{jk}}.
\end{equation}

We now calculate the off-diagonal elements. Since $\det U = 1$, then the following identity holds: 
\begin{equation}\label{eq:uoffdiag}
u_{12} u_{21} = - {g(p_{jk},p_l,p_{jkl}) \over r_{jk}^2},\\
\end{equation}
and in  $\mathcal{U}^{(0)}_{jk}$ $g(p_{jk},p_l,p_{jkl})\neq 0$. Since $P$ is unique up to left multiplication by a diagonal matrix $D\in \text{SL}_2(\complessi)$, we are allowed to fix $u_{21} = 1$. Then equation 
\eqref{eq:uoffdiag} gives us the element $u_{12}$. 

The system of equations $\Tr VU = p_{kl}$ and $\Tr \Lambda_{jk} V^{-1} U = p_{jl}$ gives us a parametrization for the off-diagonal elements of $V$:
 \begin{equation}
v_{12} = -\frac{y_{ik}-y_{ij} \lambda^{-}_{jk}}{r_{jk}^2},\qquad
v_{21} = \frac{y_{ik}-y_{ij} \lambda^{+}_{jk}}{g(p_{jk},p_{i},p_{ijk})},\\
\end{equation}
where $y_{ik}$ and $y_{ij}$ are defined in \eqref{eq:yik} and \eqref{eq:yij} respectively.
Finally, consider the system of equations $\Tr TU = p_{il}$ and $\Tr TWVU = \Tr T\Lambda_{jk} U = p_{ijkl}$, then we have the following parametrization for $t_{12}$ and $t_{21}$:
 \begin{equation}
t_{12} = -\frac{y_{il} + y_{ijkl} \lambda^{+}_{jk}}{r_{jk}^2}, \qquad
t_{21} = \frac{y_{il} + y_{ijkl} \lambda^{-}_{jk}}{g(p_{jk},p_{l},p_{jkl})},
\end{equation}
where $y_{il}$ and $y_{ijkl}$ are defined in \eqref{eq:yil} and \eqref{eq:yijkl} respectively.
This concludes the proof.
\endproof

Theorem \ref{lm:coord} shows that $(p_1,\dots,p_4,p_{21},\dots,p_{43},p_{321},\dots,p_{421})$ parameterize  the following open subset of $\mathcal U$:
\begin{equation}
\bigcup_{j>k} \mathcal U_{jk}^{(0)}\cup \mathcal U_{jk}^{(1)}\cup  \mathcal U_{jk}^{(2)}\label{eq:bigU}
\end{equation}
We now show that  it is possible to parameterize the monodromy matrices in terms of  
$p\in\mathcal A$ also outside of this open subset.

\begin{lemma}\label{lm:coord1}
Let $(M_1,\dots,M_4) \in \mathcal U$ and $p\in\mathcal A$. Assume that  $p_{jk}\neq\pm 2$ for at least one choice of $j\neq k$, $j,k=1,\dots,4$ and
\begin{equation}\label{eq:chatr-3}
g( p_{jk}, p_{l}, p_{jkl} )=g(p_{j},p_{k},p_{jk})= g(p_{jk},p_{i},p_{ijk})=0,
\end{equation}
where  $g(x,y,z) := x^2+y^2+z^2-xyz-4$,
then there exists at least an index $l$ for which $p_{lk}\neq \lambda_l \lambda_k+\frac{1}{ \lambda_l \lambda_k}$ and  a global conjugation $P\in SL_2(\mathbb C)$ such that:
\begin{eqnarray}
&&
PM_kP^{-1}=\left(\begin{array}{cc}\lambda_k& 1\\ 0 &\frac{1}{\lambda_k}\\ \end{array}
\right), \qquad
PM_jP^{-1}=\left(\begin{array}{cc}\lambda_j& -\lambda_j \lambda_k\\ 0 &\frac{1}{\lambda_j}\\ \end{array}
\right), \label{eq:param-g0-j}\\
&&
PM_lP^{-1}=\left(\begin{array}{cc}\lambda_l& 0\\ p_{lk}-\lambda_l \lambda_k-\frac{1}{ \lambda_l \lambda_k} &\frac{1}{\lambda_l}\\ \end{array}
\right), \label{eq:param-g0-l}\\
&&
PM_iP^{-1}=\left\{\begin{array}{l}\left(\begin{array}{cc}\lambda_i& 0\\ p_{ik}-\lambda_i \lambda_k-\frac{1}{ \lambda_i \lambda_k} &\frac{1}{\lambda_i}\\ \end{array}
\right),\quad\hbox{ for } p_{il}=\lambda_i \lambda_l+\frac{1}{ \lambda_i \lambda_l},\\
\left(\begin{array}{cc}\lambda_i& \frac{p_{il}-\lambda_i \lambda_l-\frac{1}{ \lambda_i \lambda_l}}{p_{lk}-\lambda_l \lambda_k-\frac{1}{ \lambda_l \lambda_k}}
 \\ 0 &\frac{1}{\lambda_i}\\ \end{array}
\right),\quad\hbox{ for } p_{il}\neq \lambda_i \lambda_l+\frac{1}{ \lambda_i \lambda_l},
\end{array}\right. \label{eq:param-g0-i}
\end{eqnarray}
where $\lambda_s+\frac{1}{\lambda_s}=p_s$, $\forall s=1,\dots,4$.
\end{lemma}

\proof Proceeding as in the proof of Theorem \ref{lm:coord}, we bring the product matrix $M_j M_k$ into the diagonal form. Condition \eqref{eq:chatr-3} implies that
the following equations must be satisfied (we have absorbed the global conjugation $P$ in the matrices $M_1,\dots,M_4$, here):
$$
(M_1)_{12}(M_1)_{21}=(M_2)_{12}(M_2)_{21}=(M_3)_{12}(M_3)_{21}=(M_4)_{12}(M_4)_{21}=0.
$$
By global conjugation by a permutation matrix, we can assume that $(M_k)_{12}\neq 0$ and then by global diagonal conjugation we can put $M_k$ in Jordan normal form. Then, since $M_j= \Lambda_{jk}M_k^{-1}$ we immediately obtain  \eqref{eq:param-g0-j}.
Since the monodromy group must be irreducible, one of the two remaining matrices, call it $M_l$, must have non zero $21$ entry. Then since $\Tr(M_l M_k)=p_{lk}$, we obtain $(M_l)_{21}= p_{lk}-\lambda_l \lambda_k-\frac{1}{ \lambda_l\lambda_k}\neq 0$, and therefore \eqref{eq:param-g0-l}.

Now if the last matrix is also lower triangular, by imposing $\Tr M_i M_k=p_{ik} $, we obtain the first formula in \eqref{eq:param-g0-i}, and it is immediate to check that then
$p_{il}= \lambda_i \lambda_l+\frac{1}{ \lambda_i \lambda_l}$. Otherwise, if $M_i$ is upper triangular, by imposing $\Tr M_i M_l=p_{il} $, we obtain the second formula  \eqref{eq:param-g0-i}, and it is immediate to check that then  $p_{il}\neq \lambda_i \lambda_l+\frac{1}{ \lambda_i \lambda_l}$.
\endproof

\begin{proposition}
\label{prop:coord}
Let $(M_1,\dots,M_4) \in \mathcal U$ and  $p\in\mathcal A$. Assume that  $p_{jk}= 2 \epsilon_{jk}$ for all $j,k=1,\dots,4$, where $\epsilon_{jk}=\pm 1$. Then, if that at least one matrix $M_i$ is diagonalizable there exists a choice of the ordering of  the indices $i,j,k,l\in\{1,2,3,4\}$ such that the following parameterization holds true:
\begin{equation}\label{pij2-param-i}
PM_iP^{-1}=\left(\begin{array}{cc}\lambda_i&0\\ 0&\frac{1}{\lambda_i}\\\end{array}\right),\qquad \lambda_i\neq \pm 1, \quad \lambda_i+\frac{1}{\lambda_i}=p_i,
\end{equation}
\begin{equation}\label{pij2-param-k}
PM_kP^{-1}=\left(\begin{array}{cc} -\frac{p_k- 2\epsilon_{ki}\lambda_i}{\lambda_i^2-1}&  -\frac{(p_k\lambda_i-\epsilon_{ki}  (\lambda_i^2+1))^2}{(\lambda_i^2-1)^2}\\
1 & \frac{\lambda_i(p_k \lambda_i- 2\epsilon_{ki})}{\lambda_i^2-1} \\
\end{array}\right),
\end{equation}
and for $p_k\neq \epsilon_{ki} p_i$:
\begin{eqnarray}\label{pij2-param-j}
PM_jP^{-1}&=&\left(\begin{array}{c} -\frac{p_j- 2\epsilon_{ji}\lambda_i}{\lambda_i^2-1}\\
  \frac{(\lambda_i^2-1)(2\epsilon_{kj}-p_{ikj}\lambda_i)+(p_k \lambda_i- 2\epsilon_{ki})(p_j \lambda_i- 2\epsilon_{ji}) }{(p_k\lambda_i-\epsilon_{ki}  (\lambda_i^2+1))^2} \\
\end{array}\right.\qquad\qquad\\
&&\quad\quad \left. 
\begin{array}{c}  - \frac{\lambda_i^2(p_k \lambda_i- 2\epsilon_{ki})(p_j \lambda_i- 2\epsilon_{ji}) +\lambda_i(\lambda_i^2-1)(p_{ikj}- 2\epsilon_{kj}\lambda_i)}{(\lambda_i^2-1)^2}\\
 \frac{\lambda_i(p_j \lambda_i- 2\epsilon_{ji})}{\lambda_i^2-1} \\
\end{array}
\right)\nonumber
\end{eqnarray}
\begin{eqnarray}\label{pij2-param-l}
PM_lP^{-1}&=&\left(\begin{array}{c} -\frac{p_l- 2\epsilon_{li}\lambda_i}{\lambda_i^2-1}\\
  \frac{(\lambda_i^2-1)(2\epsilon_{kl}-p_{ikl}\lambda_i)+(p_k \lambda_i- 2\epsilon_{ki})(p_l \lambda_i- 2\epsilon_{li}) }{(p_k\lambda_i-\epsilon_{ki}  (\lambda_i^2+1))^2} \\
\end{array}\right.\qquad\qquad\\
&& \quad \quad \left. 
\begin{array}{c}  - \frac{\lambda_i^2(p_k \lambda_i- 2\epsilon_{ki})(p_l \lambda_i- 2\epsilon_{li}) +\lambda_i(\lambda_i^2-1)(p_{ikl}- 2\epsilon_{kl}\lambda_i)}{(\lambda_i^2-1)^2}\\
 \frac{\lambda_i(p_l \lambda_i- 2\epsilon_{li})}{\lambda_i^2-1} \\
\end{array}
\right)\nonumber
\end{eqnarray}
and if $p_k=\epsilon_{ki} p_i$, then  $p_{ikj}(\lambda_i^2+1)\neq 2 \lambda_i(\epsilon_{ki}\epsilon_{ji}+\epsilon_{kj})$ and $p_{ikl}(\lambda_i^2+1)\neq 2 \lambda_i(\epsilon_{ki}\epsilon_{li}+\epsilon_{kl} )$ and
\begin{eqnarray}\label{pij2-param-jeq}
PM_jP^{-1}&=&\left(\begin{array}{c} \frac{\lambda_i(p_{ikj} \lambda_i-2 \epsilon_{kj})} {\epsilon_{ki}(\lambda_i^2-1)}\\
\\
\frac{\lambda_i^4(p_{ikj}\lambda_i-2\epsilon_{kj})^2-2 \epsilon_{kj}\epsilon_{ji}\lambda_i^2(p_{ikj}\lambda_i-2\epsilon_{kj})(\lambda_i^2-1)+(\lambda_i^2-1)^2}
{(\lambda_i^2-1)^2\lambda_i(2\lambda_i(\epsilon_{ki}\epsilon_{ji}+\epsilon_{kj})-p_{ikj}(\lambda_i^2+1)^2)}
\end{array}\right.\\
&& \quad\qquad\qquad \quad \left. 
\begin{array}{c}  \lambda_i(p_{ikj}(\lambda_i^2+1)-2 \lambda_i(\epsilon_{ki}\epsilon_{ji}+\epsilon_{kj} ))\\
 \frac{2 \epsilon_{ki}\epsilon_{ji}\lambda_i(\lambda_i^2-1)-\lambda_i^3(p_{ikj}\lambda_i-2\epsilon_{kj})}{\epsilon_{ki}(\lambda_i^2-1)}\\
\end{array}
\right)\nonumber
\end{eqnarray}
\begin{eqnarray}\label{pij2-param-leq}
PM_lP^{-1}&=&\left(\begin{array}{c} \frac{\lambda_i(p_{ikl} \lambda_i-2 \epsilon_{kl})} {\epsilon_{ki}(\lambda_i^2-1)}\\
\\
\frac{\lambda_i^4(p_{ikl}\lambda_i-2\epsilon_{kl})^2-2 \epsilon_{kl}\epsilon_{li}\lambda_i^2(p_{ikl}\lambda_i-2\epsilon_{kl})(\lambda_i^2-1)+(\lambda_i^2-1)^2}
{(\lambda_i^2-1)^2\lambda_i(2\lambda_i(\epsilon_{ki}\epsilon_{li}+\epsilon_{kl})-p_{ikl}(\lambda_i^2+1)^2)}
\end{array}\right.\\
&& \quad\qquad\qquad \quad \left. 
\begin{array}{c}  \lambda_i(p_{ikl}(\lambda_i^2+1)-2 \lambda_i(\epsilon_{ki}\epsilon_{li}+\epsilon_{kl} ))\\
 \frac{2 \epsilon_{ki}\epsilon_{li}\lambda_i(\lambda_i^2-1)-\lambda_i^3(p_{ikl}\lambda_i-2\epsilon_{kl})}{\epsilon_{ki}(\lambda_i^2-1)}\\
\end{array}
\right)\nonumber
\end{eqnarray}
If none of the monodromy matrices is diagonalizable, then there exists a choice of the ordering of  the indices $i,j,k,l\in\{1,2,3,4\}$ such that the following parameterization holds true:
\begin{equation}\label{ndpij2-param-i}
PM_iP^{-1}=\left(\begin{array}{cc}\epsilon_i&1\\ 0&\epsilon_i\\
\end{array}\right),\qquad 
PM_jP^{-1}=\left(\begin{array}{cc}\epsilon_j&0\\4\epsilon_{ij} &\epsilon_j\\
\end{array}\right),
\end{equation}
\begin{equation}\label{ndpij2-param-k}
PM_kP^{-1}=\left(\begin{array}{cc}
 \frac{p_{ijk}-2 \epsilon_{ik}\epsilon_{j} -2 \epsilon_{jk}\epsilon_{i} +2 \epsilon_{i}\epsilon_{j} \epsilon_k } {4\epsilon_{ij}}
&\frac{\epsilon_{jk}-\epsilon_j\epsilon_k}{2\epsilon_{ij}}\\
2(\epsilon_{ik}-\epsilon_i\epsilon_k)& \frac{2 \epsilon_{ik}\epsilon_{j} +2 \epsilon_{jk}\epsilon_{i}+
8  \epsilon_{ij}\epsilon_{k} -2 \epsilon_{i}\epsilon_{j} \epsilon_k -p_{ijk}} {4\epsilon_{ij}} \\
\end{array}\right)
\end{equation}
\begin{equation}\label{ndpij2-param-l}
PM_lP^{-1}=\left(\begin{array}{cc}
 \frac{p_{ijl}-2 \epsilon_{il}\epsilon_{j} -2 \epsilon_{jl}\epsilon_{i} +2 \epsilon_{i}\epsilon_{j} \epsilon_l } {4\epsilon_{ij}}
&\frac{\epsilon_{jl}-\epsilon_j\epsilon_l}{2\epsilon_{ij}}\\
2(\epsilon_{il}-\epsilon_i\epsilon_l)& \frac{2 \epsilon_{il}\epsilon_{j} +2 \epsilon_{jl}\epsilon_{i}+
8  \epsilon_{ij}\epsilon_{l} -2 \epsilon_{i}\epsilon_{j} \epsilon_l -p_{ijl}} {4\epsilon_{ij}} \\
\end{array}\right)\end{equation}
\end{proposition}

\proof 
First let us assume that at least one matrix $M_i$ is diagonal and work in the basis in which $M_i$ assumes the form \eqref{pij2-param-i} with $\lambda_i\neq\pm 1$.

Let $j\neq i$, then we have a set of linear equations in the diagonal elements of $M_j$:
$$
\Tr\left(M_i M_j \right)=2 \epsilon_{ji} ,\qquad \Tr M_j=p_j, \quad \epsilon_{ji} =\pm 1,
$$
that it is solved by:
\begin{equation}\label{pij2-param-diag}
(M_j)_{11}= -\frac{p_j- 2\epsilon_{ji}\lambda_i}{\lambda_i^2-1},\qquad 
(M_j)_{22}= \frac{\lambda_i(p_j \lambda_i- 2\epsilon_{ji})}{\lambda_i^2-1},
\end{equation}
for $j=1,\dots,4$, $j\neq i$.

Since the monodromy group is not reducible, there is at least one matrix $M_k$, $k\neq i$ such that in the chosen basis, $(M_k)_{21}\neq 0$, then we use the freedom of global diagonal conjugation to set $(M_k)_{21}= 1$. Since $\det(M_k)=1$ we obtain the formula \eqref{pij2-param-k}. 

We now deal with the other two matrices  in the case $p_k\neq\epsilon_{ki} p_i$  - we only need to find the off-diagonal elements of these matrices. To this aim we use the following equations for $s=j,l$:
$$
\Tr(M_s M_k)=2\epsilon_{sk},\qquad 
\Tr(M_i M_k M_s)= p_{iks},
$$
which, combined with \eqref{pij2-param-diag} lead to
\eqref{pij2-param-j} and \eqref{pij2-param-k}. One can treat the case $p_k=\epsilon_{ki} p_i$  similarly, we omit the proof for brevity.
This concludes the proof of the first case.

To prove the second case, assume none of the matrices $M_1,\dots,M_4$ are diagonalizable, then ${\rm eigen}(M_i)=\{\epsilon_i,\epsilon_i\}$, $\forall i=1,\dots,4$, where $\epsilon_i=\pm 1$. Let us choose a global conjugation such that one of the matrices $M_i$ is in upper triangular form as in \eqref{ndpij2-param-i}.

Now, since the monodromy group is not reducible, there exists at least one $j$ such that $(M_j)_{21}\neq 0$. From 
$\Tr M_i M_j=2\epsilon_{ij}$ we have $2 \epsilon_i \epsilon_j + (M_j)_{21}=2\epsilon_{ij}$, so that $(M_j)_{21}\neq0$ implies $ \epsilon_i \epsilon_j =- \epsilon_{ij}$.
We perform a conjugation by a unipotent upper triangular matrix to impose  $(M_j)_{12}=0$, so we obtain the second equation in \eqref{ndpij2-param-i}.

For all other matrices we use $\Tr M_i M_s = 2 \epsilon_{is}$ and $\Tr M_j M_s = 2 \epsilon_{js}$, $s=k,l$ to find:
$$
(M_s)_{21} = 2(\epsilon_{is}-\epsilon_i\epsilon_s),\qquad (M_s)_{12} = \frac{\epsilon_{js}-\epsilon_j\epsilon_s}{2\epsilon_{ij}},\qquad 
$$
From $\Tr M_s=2\epsilon_s$ and $\Tr(M_i M_j M_s)=p_{ijs}$ we find the final formula \eqref{ndpij2-param-k}  for $s=k$ and \eqref{ndpij2-param-l} for $s=l$ respectively.
\endproof

In the following Theorem characterizes the space of parameters as an affine variety in the polynomial ring
\begin{equation}
\complessi [p_1,p_2,p_3,p_4,p_{21},p_{31},p_{32},p_{41},p_{42},p_{43},p_{321},p_{432},p_{431},p_{421}].
\label{eq:polyring}\end{equation}
\begin{theorem}\label{thm:6indep}
Consider $m := (M_1,\dots,M_4)\in\mathcal U$.

\noindent (i) The co-adjoint coordinates of $m$ defined in \eqref{15tuple} and \eqref{eq:pidef} belong to the zero locus of the following $15$ polynomials in the ring \eqref{eq:polyring}:
\begin{align}
f_{1}(p)&:=p_{32} p_{31} p_{21}  + p_{32}^2 + p_{31}^2 + p_{21}^2 -\label{eq:f321}\\
 & - (p_{1}  p_{321} + p_{2} p_{3})p_{32} - (p_{2} p_{321} + p_{1} p_{3}) p_{31} -\nonumber\\
 & - ( p_{3} p_{321} + p_{1} p_{2} ) p_{12} + p_{3}^2 + p_{2}^2 + p_{1}^2 + p_{321}^2 + p_{3} p_{2} p_{1} p_{321} - 4, \nonumber
\end{align}
\begin{align}
f_{2}(p)&:=p_{42} p_{41} p_{21} + p_{42}^2 + p_{41}^2 + p_{21}^2 - \label{eq:f421}\\
 & - (p_{1}  p_{421} + p_{2} p_{4})p_{42} - (p_{2} p_{421} + p_{1} p_{4}) p_{41} -\nonumber\\
 & - ( p_{4} p_{421} + p_{1} p_{2} ) p_{12} + p_{4}^2 + p_{2}^2 + p_{1}^2 + p_{421}^2 + p_{4} p_{2} p_{1} p_{421} - 4,\nonumber
\end{align}
\begin{align}
f_{3}(p) &:=p_{43} p_{41} p_{31} + p_{43}^2 + p_{41}^2 + p_{31}^2 -\label{eq:f431} \\
 & - (p_{1}  p_{431} + p_{3} p_{4})p_{43} - (p_{3} p_{431} + p_{1} p_{4}) p_{41} -\nonumber\\
 &- ( p_{4} p_{431} + p_{1} p_{3} ) p_{13} + p_{4}^2 + p_{3}^2 + p_{1}^2 + p_{431}^2 + p_{4} p_{3} p_{1} p_{431} - 4,\nonumber
\end{align}
\begin{align}
f_{4}(p)&:= p_{43} p_{42} p_{32} + p_{43}^2 + p_{42}^2 + p_{32}^2 - \label{eq:f432}\\
 & - (p_{2}  p_{432} + p_{3} p_{4})p_{43} - (p_{3} p_{432} + p_{2} p_{4}) p_{42} -\nonumber\\
 &- ( p_{4} p_{432} + p_{2} p_{3} ) p_{23} + p_{4}^2 + p_{3}^2 + p_{2}^2 + p_{432}^2 + p_{4} p_{3} p_{2} p_{432} - 4, \nonumber
\end{align}
\begin{align}
f_5(p) &:= - 2 p_{\infty} + p_{1}p_{2} p_{3} p_{4}+p_{1} p_{432}+p_{2} p_{431}+p_{3} p_{421}+p_{321} p_{4}+\nonumber\\
   &+p_{21} p_{43}+p_{32} p_{41}  -p_{1} p_{2} p_{43}
   -p_{1} p_{4} p_{32}
   -p_{2} p_{3} p_{41}
   -p_{3} p_{4} p_{21}-\nonumber\\
   & - p_{42} p_{31},\label{eq:f1}
\end{align}
\begin{align}
f_{6}(p) &:= p_{2} p_{3} p_{4} - p_{32} p_{4} - p_{21} p_{3} p_{41} + p_{321} p_{41} - p_{3} p_{42} + p_{1} p_{3} p_{421}- \nonumber\\ & - 
 p_{31} p_{421} - p_{2} p_{43} +p_{21} p_{431} + 2 p_{432} - p_{1} p_\infty,\label{eq:f6}
\end{align}
\begin{align}
f_7(p) &:= -p_{1} p_{4} + 2 p_{41} + p_{21} p_{42} - p_{2} p_{421} + p_{31} p_{43} + p_{21} p_{32} p_{43} -\nonumber\\ &- p_{2} p_{321} p_{43}  - p_{3} p_{431} - p_{21} p_{3} p_{432} + p_{321} p_{432} + p_{2} p_{3} p_\infty -\nonumber\\ &- p_{32} p_\infty,\label{eq:f2}
\end{align}
\begin{align}
f_8(p) &:= -p_{1} p_{2} p_{3} + p_{21} p_{3} + p_{2} p_{31} + p_{1} p_{32} - 2 p_{321} + p_{2} p_{41} p_{43} -\nonumber\\ &- p_{421} p_{43} - p_{2} p_{4} p_{431} + p_{42} p_{431} - p_{41} p_{432} + p_{4} p_\infty,\label{eq:f3}
\end{align}
\begin{align}
f_9(p) &:= -p_{1} p_{2} + 2 p_{21} + p_{31} p_{32} - p_{3} p_{321} + p_{41} p_{42} - p_{4} p_{421} +\nonumber\\ &+ p_{32} p_{41} p_{43} - p_{32} p_{4} p_{431} - p_{3} p_{41} p_{432} + p_{431} p_{432} + p_{3} p_{4} p_\infty -\nonumber\\ &- p_{43} p_\infty,\label{eq:f4}
\end{align}
\begin{align}
f_{10}(p) &:= -p_{1} p_{2} p_{4} + p_{21} p_{4} + p_{2} p_{41} + p_{1} p_{42} - 2 p_{421} + p_{1} p_{32} p_{43} -\nonumber\\&- p_{321} p_{43} -p_{32} p_{431} - \ p_{1} p_{3} p_{432} + p_{31} p_{432} + p_{3} p_{\infty},\label{eq:f5}
\end{align}
\begin{align}
f_{11}(p) &:= p_{1} p_{3} p_{4} - p_{31} p_{4} - p_{21} p_{32} p_{4} + p_{2} p_{321} p_{4} - p_{3} p_{41}- p_{321} p_{42} +\nonumber\\&+p_{32} p_{421} - p_{1} p_{43} + 2 p_{431} +p_{21} p_{432} - p_{2} p_{\infty},\label{eq:f7}
\end{align}
\begin{align}
f_{12}(p) &:= -p_{2} p_{4} + p_{21} p_{41} + 2 p_{42} - p_{1} p_{421} + p_{32} p_{43} - p_{321} p_{431} -\nonumber\\&- p_{3} p_{432} + p_{31} p_{\infty},\label{eq:f8}
\end{align}
\begin{align}
f_{13}(p) &:= p_{1} p_{3} - 2 p_{31} - p_{21} p_{32} + p_{2} p_{321} - p_{41} p_{43}+ p_{4} p_{431} +\nonumber\\& +
   p_{421} p_{432} - p_{42} p_{\infty},\label{eq:f9}
\end{align}
\begin{align}
f_{14}(p) &:= p_{2} p_{3} - p_{21} p_{31} - 2  p_{32} + p_{1} p_{321} - p_{21} p_{41} p_{43} -  p_{42} p_{43} + \nonumber\\&
 p_{1} p_{421} p_{43} + p_{21} p_{4} p_{431} - p_{421} p_{431} + p_{4} p_{432} - 
 p_{1} p_{4} p_{\infty} +\nonumber\\&+ p_{41} p_{\infty},\label{eq:f10}
\end{align}
\begin{align}
f_{15}(p) &:= -p_{3} p_{4} + p_{31} p_{41} + p_{21} p_{32} p_{41} - p_{2} p_{321} p_{41} + p_{32} p_{42} -\nonumber\\&- p_{1} p_{32} p_{421} + p_{321} p_{421} + 2 p_{43} - p_{1} p_{431} - p_{2} p_{432} + p_{1} p_{2} p_\infty-\nonumber\\&-p_{21}  p_\infty.\label{eq:f11}
\end{align}

\noindent (ii) For every given generic $p_1,\dots,p_4,p_\infty$, the affine variety $\mathcal A$ defined in \eqref{eq:algv} with 
$I = < f_1,\dots,f_{15} >$, is four dimensional.
\end{theorem}
\proof
To prove the relations  \eqref{eq:f321}, \dots,  \eqref{eq:f11} we use iterations of the \textit{skein relation}:
\begin{equation}
\Tr A B + \Tr A^{-1} B = \Tr A \Tr B ,~ \forall A,B \in \text{SL}_2(\complessi),
\label{eq:rule1}\end{equation} 
together with  \eqref{eq:cyclic}.

To prove statement (ii), we used Macaulay2 \cite{macaulay},  in order to compute the dimension of the affine variety defined in  \eqref{eq:algv} .
The result is that \eqref{eq:algv} has dimension four.
\endproof

\begin{corollary}
The quantities $(p_{21},\dots,p_{43},p_{321},\dots,p_{421})$
give a set of over-determined coordinates on the open subset $\mathcal U\subset  \hhmg2$ defined in \eqref{eq:bigopen}.
\end{corollary}

\proof
Thanks to Theorem \ref{lm:coord}, Lemma \ref{lm:coord1} and Proposition \ref{prop:coord} the quantities $p_i$, $p_{ij}$, $p_{ijk}$ parameterize the monodromy matrices up to global conjugation. Thanks to Theorem \ref{thm:6indep} for every fixed choice of $p_1,p_2,p_3,p_4,p_\infty$ only $4$ among the quantities $p_{ij}$, $p_{ijk}$ for $i,j,k=1,\dots,4$ are independent. This concludes the proof.\endproof

\section{Braid group action on $\mathcal{M}_{\mathcal{G}_2}$}\label{sec:G2case}

We start this section by proving Lemma \ref{lm:u2}.

\proof First we prove that  that action \eqref{eq:braidpaction} is well defined,  or in other words that the $I = <\mathcal{F}>= \{ f_1,\dots,f_{15}\}$ is invariant under the action \eqref{eq:braidpaction}. To this aim, we need to show that for each generator $\sigma_i$, $i=1,2,3$, $\sigma_i(I)= I$. We carry out the computation for $\sigma_1$ only, the other computations are similar.
$$
f_{1}(\sigma_1(p)) = f_{1}(p),\qquad 
f_{2}(\sigma_1(p)) = f_{2}(p), \qquad
f_{3}(\sigma_1(p)) = f_{4}(p),
$$
\begin{flalign*}
f_{4}(\sigma_1(p)) &= f_{3}(p) +(p_{21}p_{42}-p_2p_{421})f_{6}(p) + (p_{21}p_{431}-p_2p_{\infty})f_{11}(p)+ \nonumber\\ 
&(p_2p_{321}-p_{21}p_{32})f_{13}(p),&
\end{flalign*}
$$
f_{5}(\sigma_1(p)) = f_{1}(p) - p_{2}  f_{7}(p), \qquad
f_{6}(\sigma_1(p)) = f_{8}(p) - p_{21}  f_{2}(p),
$$
$$
f_{7}(\sigma_1(p)) = f_{3}(p) + p_{2}  f_{9}(p),\qquad f_{9}(\sigma_1(p)) = f_{5}(p) - p_{2}  f_{2}(p),
$$
\begin{flalign*}
f_{8}(\sigma_1(p)) &= f_{4}(p) - p_{42}  f_{2}(p) - p_{432}  f_{7}(p) + p_{32}  f_{9}(p),&
\end{flalign*}
$$
f_{10}(\sigma_1(p)) = -f_{7}(p),\qquad
\qquad
f_{11}(\sigma_1(p)) = f_{6}(p) - p_{21}  f_{7}(p),
$$
$$
f_{12}(\sigma_1(p)) = -  f_{2}(p),\qquad f_{14}(\sigma_1(p)) = -  f_{9}(p),
$$
$$
f_{13}(\sigma_1(p)) = - p_{21}  f_{9}(p) +  f_{10}(p),  \qquad f_{15}(\sigma_1(p)) =  f_{11}(p) + p_1  f_{7}(p).
$$
Similar formulae can be proved for all other generators of the braid group. This shows  that the action \eqref{eq:braidpaction} is well defined on $\mathcal A$.

In order to prove that $\sigma_i$ for $i=1,2,3$, defined in \eqref{eq:braidpaction}, is indeed an action of the braid group $B_4$, we recall 
that the braid group $B_n$ in Artin's presentation is given by:
\begin{align}
B_n = \bigl\langle \sigma_1,\dots,\sigma_{n-1}~|~&\sigma_i \sigma_{i+1} \sigma_i = \sigma_{i+1} \sigma_i \sigma_{i+1},~1\leq i \leq n-2,\nonumber\\
										&\sigma_i \sigma_j = \sigma_j \sigma_i,~|i-j|>1 \bigr\rangle,
\label{eq:braidgrouppres}\end{align}
so, we need to prove that the following relations are satisfied:
\begin{equation}\label{eq:braidrelations}
\sigma_1 \sigma_3(m) = \sigma_3 \sigma_1(m),\quad
\sigma_1 \sigma_2 \sigma_1(m) = \sigma_2 \sigma_1 \sigma_2(m),\quad
\sigma_2 \sigma_3 \sigma_2(m) = \sigma_3 \sigma_2 \sigma_3(m).\
\end{equation}
The first relation is straightforward, while the last two follow from the fact that polynomials \eqref{eq:f7},\eqref{eq:f8},\eqref{eq:f9} and \eqref{eq:f10} are zero for every $p\in
\mathcal A$.
\endproof

This lemma allows us to reformulate our classification problem as follows:

Classify all finite orbits: 
\begin{equation*}
\mathcal{O}_{P_4}(p) = \left\{ \beta(p) | \beta \in P_4 \right\},
\end{equation*}
where $p$ is the following $15$-tuple of complex quantities:
\begin{equation*}
p = (p_1,p_2,p_3,p_4,p_{\infty},p_{21},p_{31},p_{32},p_{41},p_{42},p_{43},p_{321},p_{432},p_{431},p_{421})\in \complessi^{15},
\end{equation*}
defined in \eqref{eq:pidef}, and $P_4$ is the pure braid group $P_4=\langle\beta_{21 },\beta_{31 },\beta_{32 },\beta_{41 },\beta_{42},\beta_{43} \rangle$ where:
\begin{eqnarray}
\beta_{21 } = \sigma_1^2,\qquad 
\beta_{31 } = \sigma_2^{-1}\sigma_1^2 \sigma_2,\qquad
\beta_{32 } = \sigma_2^{2},\label{mypuregens} \\
\beta_{41 } = \sigma_3^{-1} \sigma_2^{-1} \sigma_1^2 \sigma_2 \sigma_3,\qquad
\beta_{42 } = \sigma_3^{-1} \sigma_2^2 \sigma_3,\qquad
\beta_{43 } = \sigma_3^2,\nonumber
\end{eqnarray}
and the generators satisfy the following relations:
\begin{equation}
\beta_{rs}^{}\beta_{ij}\beta_{rs}^{-1} = 
\begin{cases}
\beta_{ij}, & \mbox{if } j < s < r < i, \\
 	   & \mbox{or } s < r < j < i,\\
\beta_{rj}^{-1}\beta_{ij}\beta_{rj}^{}, &s < j = r < i,\\
\beta_{rj}^{-1}\beta_{sj}^{-1}\beta_{ij}\beta_{sj}^{}\beta_{rj}^{},&j = s < r < i,\\
\beta_{rj}^{-1}\beta_{sj}^{-1}\beta_{rj}^{}\beta_{sj}^{}\beta_{ij}\beta_{sj}^{-1}\beta_{rj}^{-1}\beta_{sj}^{}\beta_{rj}^{},&s < j < r < i.
\end{cases}
\label{puregensrelations}\end{equation}

\section{Restrictions}\label{sec:method}
In this section, we select subgroups $H\subset P_4$  such that the restricted action is isomorphic to the action of the pure braid group $P_3$ on the quotient space \eqref{eq:Mmp6}.

\begin{theorem}\label{thm:match}
The following four subgroups $H_i \subset P_4$ with $i=1,\dots,4$:
$$ 
H_1 := <\beta_{32 },\beta_{43 },\beta_{42 }>, \qquad
H_2 = <\beta_{43 },\beta_{31 },\beta_{41 }>,
$$
$$
H_3 = <\beta_{21 },\beta_{42 },\beta_{41 }>,\qquad
H_4 = <\beta_{21 },\beta_{32 },\beta_{31 }>,
$$
where the generators $\beta_{jk}$, $1\leq k<j\leq 4$, are defined in \eqref{mypuregens}, are isomorphic to the  pure braid group $P_3$. Moreover
given any ordered $4$-tuple of matrices $(M_1,M_2,M_3,M_4) \in \mathcal{U}$, each $H_i$, for $i=1,\dots,4$, acts as pure braid group $P_3$ on a certain triple of matrices $(N_1,N_2,N_3)\in\widehat{\mathcal{M}}_{PVI}$  given by:
\begin{align}
&H_1:~ \widehat{N}_{1} = M_2,~\widehat{N}_{2} = M_3,~\widehat{N}_{3} = M_4,~\widehat{N}_{\infty} = (M_4 M_3 M_2)^{-1},\label{ihatn}\\
&H_2:~ \widebar{N}_{1} = M_1,~\widebar{N}_{2} = M_3,~\widebar{N}_{3} = M_4,~\widebar{N}_{\infty} = (M_4 M_3 M_1)^{-1},\label{ibarn}\\
&H_3:~ \widecheck{N}_{1} = M_1,~\widecheck{N}_{2} = M_2,~\widecheck{N}_{3} = M_4,~\widecheck{N}_{\infty} = (M_4 M_2 M_1)^{-1},\label{icheckn}\\
&H_4:~ \widetilde{N}_{1} = M_1,~\widetilde{N}_{2} = M_2,~\widetilde{N}_{3} = M_3,~\widetilde{N}_{\infty} = (M_3 M_2 M_1)^{-1}.\label{itilden}
\end{align}
\end{theorem}
\begin{proof}
To prove that  each $H_i$, $i=1,\dots,4$ is isomorphic to $P_3$ we need to prove that its generators satisfy the relations 
\eqref{puregensrelations}, for $n=3$. This can be checked by direct computations. 

We now prove the second statement explicitly for the subgroup $H_1$, for the other subgroups a similar proof applies. First of all, thanks to \eqref{eq:cyclic} we have immediately:
\begin{equation*}
\widehat{N}_\infty \widehat{N}_3 \widehat{N}_2 \widehat{N}_1 = \mathbb{I},
\end{equation*}
so that   $\widehat{n}=(\widehat{N}_{1},\widehat{N}_{2},\widehat{N}_{3})\in\widehat{\mathcal{M}}_{PVI}$. 
To show that the subgroup $H_1$ acts as pure braid group $P_3$ on $\widehat{\mathcal{M}}_{PVI}$  we use the fact that the generators of $H_1$ are defined in terms of generators $\sigma_2$ and $\sigma_3$ of the full braid group $B_4$, so it is enough to prove that $\sigma_2$ and $\sigma_3$ act as generators of the braid group $B_3$ on  $\widehat{\mathcal{M}}_{PVI}$. Consider \eqref{ihatn}, then the following relations hold:
\begin{align}
\sigma_2(m) = (M_1,M_3,M_3 M_2 M_3^{-1} ,M_4)= (\widehat{N}_{2},\widehat{N}_{2}\widehat{N}_{1}\widehat{N}_{2}^{-1},\widehat{N}_{3}) &= \sigma^{(PVI)}_1(\widehat{n}),\nonumber\\
\sigma_3(m) = (M_1,M_2,M_4,M_4 M_3 M_4^{-1} )= (\widehat{N}_{1},\widehat{N}_{3},\widehat{N}_{3}\widehat{N}_{2}\widehat{N}_{3}^{-1}) &= \sigma^{(PVI)}_2(\widehat{n}).\label{eq:ident}
\end{align}
This concludes the proof.
\end{proof}

We now consider the action of the subgroups $H_i$ for $i=1,\dots,4$ in terms of co-adjoint coordinates on $\widehat{\mathcal{M}}_{PVI}$:
\begin{align}
&\widehat{q}:=(\widehat{q}_1,\widehat{q}_2,\widehat{q}_3,\widehat{q}_{\infty},\widehat{q}_{21},\widehat{q}_{31},\widehat{q}_{32}),\quad
\widebar{q}:=(\widebar{q}_1,\widebar{q}_2,\widebar{q}_3,\widebar{q}_{\infty},\widebar{q}_{21},\widebar{q}_{31},\widebar{q}_{32}),\nonumber\\
&\widecheck{q}:=(\widecheck{q}_1,\widecheck{q}_2,\widecheck{q}_3,\widecheck{q}_{\infty},\widecheck{q}_{21},\widecheck{q}_{31},\widecheck{q}_{32}),
\quad
\widetilde{q}:=(\widetilde{q}_1,\widetilde{q}_2,\widetilde{q}_3,\widetilde{q}_{\infty},\widetilde{q}_{21},\widetilde{q}_{31},\widetilde{q}_{32}),
\label{eq:subq}\end{align}
where $\widehat{q}_i = \Tr \widehat{N}_i$ for $i=1,2,3,\infty$ and $\widehat{q}_{j k} = \Tr \widehat{N}_j\widehat{N}_k$ for $j>k,~j,k=1,2,3$ and similar formulae for $\widebar q$, 
$\widecheck{q}$ and $\widetilde{q}$.
Then identifications \eqref{ihatn}-\eqref{itilden} imply the identities summarized in Table \ref{tb:matching}, 
where $p_i$,$p_{ij}$,$p_{ijk}$ are defined in \eqref{eq:pidef} and elements in the same column are identical.
\begin{table}[!ht]
\resizebox{\textwidth}{!}{
\begin{tabular}{ | c || c | c | c | c | c | c | c | c | c | c | c | c | c | c | c |}
\hline
\rule{0pt}{2.5ex}\rule[-1.2ex]{0pt}{0pt} & $p_1$ & $p_2$ & $p_3$ & $p_4$ & $p_{\infty}$ & $p_{21}$ & $p_{31}$ & $p_{32}$ & $p_{41}$ & $p_{42}$ & $p_{43}$ & $p_{321}$ & $p_{432}$ & $p_{431}$ & $p_{421}$\\
\hline\hline
\rule{0pt}{2.5ex}\rule[-1.2ex]{0pt}{0pt} $H_1$ & & $\widehat{q}_1$ &  $\widehat{q}_2$  &$\widehat{q}_3$   &   &   &   & $\widehat{q}_{21}$  &   &$\widehat{q}_{31}$   &$\widehat{q}_{32}$   &   & $\widehat{q}_{\infty}$   &  &  \\
\hline
\rule{0pt}{2.5ex}\rule[-1.2ex]{0pt}{0pt} $H_2$ & $\widebar{q}_1$  &  & $\widebar{q}_2$  & $\widebar{q}_3$  &   &   & $\widebar{q}_{21}$  &   & $\widebar{q}_{31}$  &   & $\widebar{q}_{32}$  &   &   & $\widebar{q}_{\infty}$  &  \\
\hline
\rule{0pt}{2.5ex}\rule[-1.2ex]{0pt}{0pt} $H_3$ & $\widecheck{q}_1$ &$\widecheck{q}_2$ &   &$\widecheck{q}_3$   &   & $\widecheck{q}_{21}$  &   &   & $\widecheck{q}_{31}$  &$\widecheck{q}_{32}$   &   &   &   &   & $\widecheck{q}_{\infty}$ \\
\hline
\rule{0pt}{2.5ex}\rule[-1.2ex]{0pt}{0pt} $H_4$ & $\widetilde{q}_1$ &$\widetilde{q}_2$ &$\widetilde{q}_3$   &&   &$\widetilde{q}_{21}$   &$\widetilde{q}_{31}$   &$\widetilde{q}_{32}$   &   &   &   & $\widetilde{q}_{\infty}$  &   &   &    \\
\hline
\end{tabular}
}\\ 
\caption {Matching using traces: elements on the same columns must be equal.}\label{tb:matching}
\end{table}

We define the following four projections:
\begin{equation}
\widetilde{\pi},\widehat{\pi},\widecheck{\pi},\widebar{\pi} : \mathcal A \mapsto \widehat{\mathcal{M}}_{PVI},
\label{projections}\end{equation}
as follows
\begin{equation}\label{projq}
\begin{split}
\widetilde{\pi}(p) &:= (p_1,p_2,p_3,p_{321},p_{21},p_{31},p_{32}) = \widetilde{q},\\
\widehat{\pi}(p) &:= (p_2,p_3,p_4,p_{432},p_{32},p_{42},p_{43}) = \widehat{q},\\
\widecheck{\pi}(p) &:= (p_1,p_2,p_4,p_{421},p_{21},p_{41},p_{42}) = \widecheck{q},\\
\widebar{\pi}(p) &:= (p_1,p_3,p_4,p_{431},p_{31},p_{41},p_{43}) = \widebar{q}.
\end{split}
\end{equation}
Viceversa, given four $7$-ples  $\widetilde{q}$,$\widehat{q}$,$\widecheck{q}$,$\widebar{q}$, such that they satisfy the equalities in the columns of Table \ref{tb:matching}, we can lift them to a point $p\in\mathcal A$, in which the value of $p_\infty$ can be recovered using relation \eqref{eq:f5}. We call this {\it matching procedure.}

\section{Input set}

The classification result by Lisovyy and Tykhyy produced a list of all finite orbits under the action of  the braid group $B_3$ modulo the action of the group ${F_4}$ of Okamoto transformations acting on \mp6. However, points $q$ that are equivalent modulo the action of the group ${F_4}$ of Okamoto transformations, and of the pure braid group $P_3$, don't necessarily produce candidate points $p$ that are equivalent modulo the action of the symmetry group $\text{G}$ of the Garnier system $\mathcal G_2$ nor by the action of the pure braid group $P_4$. Therefore, we need to expand the list of input points $q$ by considering all images under  ${F_4}$ and  $P_3$. In this section, we define an {expansion algorithm} that applies  the action of ${F_4}$ and  $P_3$ to the $45$ exceptional orbits of \cite{LT}. Thanks to the fact that 
the action of ${F_4}$ over these $45$ finite orbits is finite,
the result is a finite set that we call $E_{45}$. This set does not include points that correspond to solutions of Okamoto type nor solutions corresponding to Picard or  Kitaev orbits - we will deal with these points in Section  \ref{se:mp}.

\subsection{The classification result by Lisovyy and Tykhyy}
In order to expand Lisovyy and Tykhyy  list of $45$  finite orbits (see Table 5 in \cite{LT})   it is best to introduce the following quantities:
\begin{equation}
\begin{array}{ll}
\omega_{1} := q_{1} q_{\infty} + q_{3} q_{2},\qquad
\omega_{2} := q_{2} q_{\infty} + q_{3} q_{1},\qquad
\omega_{3} := q_{3} q_{\infty} + q_{2} q_{1},\\
\omega_{4} := q_{3}^2 + q_{2}^2 + q_{1}^2 + q_{\infty}^2 + q_{3} q_{2} q_{1} q_{\infty}.
\end{array}
\label{omegasPVI}\end{equation}

The group  ${F_4}$ of Okamoto transformations of the sixth Painlev\'e equation acts as $K_4\rtimes S_3$ on $(\omega_1,\dots,\omega_4)$ \cite{LT}. 
Extending this action to the $q_{ij}$s, namely acting on $(\omega_1,\dots,\omega_4,q_{21},q_{31},q_{32})$ it is straightforward to prove the following:
\begin{proposition}\label{thm:oktransf}
The group $F_4$ of the Okamoto transformations of the sixth Painlev\'e equation is generated by the following transformations that act on $(\omega_1,\dots,\omega_4,q_{21},q_{31},q_{32})$ as follows:
\begin{align}
&s_i (q_{21},q_{31},q_{32},\omega_1,\omega_2,\omega_3,\omega_4) = (q_{21},q_{31},q_{32},\omega_1,\omega_2,\omega_3,\omega_4),~i=1,2,3,\infty,\delta ,\nonumber\\
&r_1 (q_{21},q_{31},q_{32},\omega_1,\omega_2,\omega_3,\omega_4)= (-q_{21},-q_{31},q_{32},\omega_1,-\omega_2,-\omega_3,\omega_4) ,\nonumber\\
&r_2 (q_{21},q_{31},q_{32},\omega_1,\omega_2,\omega_3,\omega_4) = (-q_{21},q_{31},-q_{32},-\omega_1,\omega_2,-\omega_3,\omega_4) ,\nonumber\\
&r_3 (q_{21},q_{31},q_{32},\omega_1,\omega_2,\omega_3,\omega_4) = (q_{21},-q_{31},-q_{32},-\omega_1,-\omega_2,\omega_3,\omega_4) ,\nonumber\\
&P_{13} (q_{21},q_{31},q_{32},\omega_1,\omega_2,\omega_3,\omega_4)= (q_{32},\omega_2 - q_{31} - q_{21}q_{32},q_{21},\omega_3,\omega_2,\omega_1,\omega_4),\nonumber\\
&P_{23} (q_{21},q_{31},q_{32},\omega_1,\omega_2,\omega_3,\omega_4) = (\omega_2 - q_{31} - q_{21}q_{32},q_{21},q_{32},\omega_1,\omega_3,\omega_2,\omega_4).\nonumber
\end{align}
\end{proposition}

\proof The proof of this is a consequence of the results of \cite{inabaiwasakisaito} and \cite{MazzoccoCanonical}.\endproof

In particular we observe that  $P_{13}$ and $P_{23}$ are elements of the braid group $B_3$ - since we act only on points that have finite orbits under the action of the braid group, the action of the whole group $F_4$ produces a finite set of values. 
All these values will be in the form $(\omega_1,\dots,\omega_4,q_{21},q_{31},q_{32})$;  in order to extract $q_1,q_2,q_3$ and $q_\infty$ we use the fact that we can consider the 
relations  \eqref{omegasPVI}  as a system of equations in $q_1,q_2,q_3$ and $q_\infty$ and that each $q_i$ has the form:
$$
q_i=2\cos\pi\theta_i,\quad i=1,2,3,\infty.
$$
One particular solution of  equations  \eqref{omegasPVI}  is listed in \cite{LT} in terms of $\theta_1,\theta_2,\theta_3,\theta_\infty$ for each point in the Table 
5 in \cite{LT}. We can then compute all other solutions $q_1,q_2,q_3$ and $q_\infty$ by using the following:

\begin{lemma}\label{lm:3}
Suppose $\omega_1,\omega_2,\omega_3,\omega_4$ are given and consider system \eqref{omegasPVI} in the variables $q_1,q_2,q_3,q_\infty$, then this system admits at most $24$ solutions. Any two such solutions are related by the following elements of $F_4$:
\begin{eqnarray}
&& id,\qquad
 \alpha,\qquad
 \beta,\qquad  \gamma,\qquad  \alpha\cdot\beta,\qquad \alpha\cdot\gamma,\qquad  \beta\cdot\gamma,\nonumber\\
&& \alpha \cdot\beta\cdot \gamma\qquad
 s_\delta,\qquad
 \alpha\cdot s_\delta,\qquad
 \beta\cdot s_\delta,\qquad
 \gamma\cdot s_\delta,\qquad
 \alpha \cdot\beta \cdot s_\delta,\nonumber\\ 
&& \alpha\cdot \gamma\cdot s_\delta,\qquad
\beta\cdot \gamma\cdot s_\delta,\qquad \alpha \cdot\beta\cdot \gamma\cdot s_\delta,\qquad
s_\delta s_1,\qquad  \alpha\cdot s_\delta\cdot s_1,\nonumber\\
&&
 \beta\cdot s_\delta\cdot s_1,\qquad \gamma\cdot s_\delta\cdot s_1
 \qquad \alpha \cdot\beta\cdot s_\delta\cdot s_1,\quad
 \alpha \cdot\gamma\cdot s_\delta \cdot s_1,\nonumber\\
&&
 \beta\cdot\gamma\cdot s_\delta\cdot s_1,\qquad \alpha \cdot\beta\cdot\gamma\cdot s_\delta\cdot s_1.\nonumber
\end{eqnarray}
where $\alpha$, $\beta$, $\gamma$, $s_\delta$, $s_1$ act as follows on the parameters $\theta_i$:
$$
\alpha(\theta_1,\theta_2,\theta_3,\theta_\infty)=(1+\theta_1,1+\theta_2,1+\theta_3,1+\theta_\infty)
$$
$$
\beta(\theta_1,\theta_2,\theta_3,\theta_\infty)=(\theta_2,\theta_1,\theta_\infty-2,\theta_3),\qquad
\gamma(\theta_1,\theta_2,\theta_3,\theta_\infty)=(\theta_3,\theta_\infty-2,\theta_1,\theta_2)
$$
$$
s_\delta(\theta_1,\theta_2,\theta_3,\theta_\infty)=(\theta_1-\delta,\theta_2-\delta,\theta_3-\delta,\theta_\infty-\delta),\quad \delta=\frac{\theta_1+\theta_2+\theta_3+\theta_\infty}{2},
$$
$$
s_1(\theta_1,\theta_2,\theta_3,\theta_\infty)=(-\theta_1,\theta_2,\theta_3,\theta_\infty).
$$
\end{lemma}

\proof It is an immediate consequence of Proposition 10 in \cite{LT}. \endproof

This lemma allows us to calculate all the solutions of the system \eqref{omegasPVI} in terms of the given  
$\omega_1,\omega_2,\omega_3,\omega_4$ starting from only one solution $q_1,q_2,q_3$ and $q_\infty$. We are therefore able to set up our expansion algorithm:

\begin{al}\label{al:004}\quad For every line of Table 5 in \cite{LT}, take the values $(\omega_1,\dots,\omega_4,$ $q_{21},q_{31},q_{32})$ and the corresponding $(q_1,q_2,q_3,q_\infty)$ given in \cite{LT}.

\begin{enumerate}
\item Apply to $(\omega_1,\dots,\omega_4,q_{21},q_{31},q_{32})$ all 48 transformations of the group $K_4\rtimes S_3$. For each new set of values $(\omega_1',\dots,\omega_4',q_{21}',q_{31}',q_{32}')$  obtained in this way, compute the corresponding  $(q_1',\dots,q_\infty')$ as the result of the same transformation on $(q_1,q_2,q_3,q_\infty)$.
\item For every element $(\omega_1',\dots,\omega_4',q_{21}',q_{31}',q_{32}')$ obtained in step 1, generate their orbit under the action of the braid group $B_3$. For each new set of values $(\omega_1'',\dots,\omega_4'',q_{21}'',q_{31}'',q_{32}'')$  obtained in this way, compute the corresponding  $(q_1'',\dots,q_\infty'')$ as the result of the same braid on $(q_1',q_2',q_3',q_\infty')$.
\item  For every element $(\omega_1'',\dots,\omega_4'',q_{21}'',q_{31}'',q_{32}'')$ and $(q_1'',\dots,q_\infty'')$ obtained in step 2,
find all other solutions $(q_1''',q_2''',q_3''',q_\infty''')$ of the system \eqref{omegasPVI} for $(\omega_1'',\dots,\omega_4'')$ by applying the transformations in Lemma \ref{lm:3}  to $(q_1'',\dots,q_\infty'')$.
\item Merge $(q_1''',q_2''',q_3''',q_\infty'')$ and $(q_{21}'',q_{31}'',q_{32}'')$ into:
$$q'''=(q_1''',q_2''',q_3''',q_\infty''',q_{21}'',q_{31}'',q_{32}'').$$
\item Generate the $P_3$-orbit of $q'''$ and save the result in the set $\text{E}_{45}$
\end{enumerate}
\end{al}
Once this algorithm ends, the set $\text{E}_{45}$ will contain only a finite number of orbits. This set contains $86,768$ points.

\section{Matching procedure}\label{se:mp}
In this Section, we propose a procedure to construct all \textit{candidate} points $p \in \mathcal A$: 

\begin{definition}
A point $p$ such that its four projections 
$\widehat{q},\widecheck{q},\widebar{q},\widetilde{q}$, defined in \eqref{projq}, generate finite orbits under the action of $P_3$ and such that at most one projections is a Picard or Kitaev orbit, is said to be a \textit{candidate} point. 
\end{definition}

Note that, to generate a \textit{candidate} point $p$, it is not necessary to know all four projections $\widehat{q},\widecheck{q},\widebar{q},\widetilde{q}$. Indeed, looking at Table \ref{tb:matching}, we see that if we give three projections, then only the value of $p_\infty$ and one value $p_{ijk}$ will be undetermined, but we can calculate these values from  \eqref{eq:f1} and by choosing appropriately one of the four relations $f_{1},\dots,f_{4}$, defined in \eqref{eq:f321}-\eqref{eq:f432} respectively. 
So, in order to obtain the set $\mathcal{C}$ of all \textit{candidate} points, we can set up three matching procedures, each of them based on the knowledge of only $3$ projections. We denote by $\widetilde C$,  $\widehat{C}$, $\widecheck{C}$ and $\widebar{C}$ the sets obtained by matching three projections and missing $\widetilde{q}$, $\widehat{q}$, $\widecheck{q}$ or $\widetilde{q}$ respectively.

In order to construct the set $\mathcal{C}$, the union of all the above four sets $\widetilde{\mathcal{C}},\widehat{\mathcal{C}},\widecheck{\mathcal{C}},\widebar{\mathcal{C}}$ must be taken:
\begin{equation}
\mathcal{C} = \widetilde{\mathcal{C}}\cup \widehat{\mathcal{C}}\cup \widecheck{\mathcal{C}}\cup \widebar{\mathcal{C}}.
\label{bigset}\end{equation}
As we are going to show in the next Lemma, it is enough to know only one of the sets $\widetilde{\mathcal{C}},\widehat{\mathcal{C}},\widecheck{\mathcal{C}},\widebar{\mathcal{C}}$ to generate the whole set $\mathcal{C}$:

\begin{lemma}\label{lm:pperms}
Consider $m \in \mathcal U$ and the permutation $\pi_{(1 2 3 4)}$ that acts
on the co-adjoint coordinates of $m$ as follows:
\begin{equation*}
\pi_{(1 2 3 4)}(p)= (p_4,p_1,p_2,p_3,p_\infty,p_{41},p_{42},p_{21},p_{43},p_{31},p_{32},p_{421},p_{321},p_{432},p_{431}),\\
\end{equation*} 
then:
\begin{equation}
\pi_{(1 2 3 4)}(\widetilde{\mathcal{C}})=\widecheck{\mathcal{C}},\quad
\pi_{(1 2 3 4)}(\widecheck{\mathcal{C}})=\widebar{\mathcal{C}}, \quad
\pi_{(1 2 3 4)}(\widebar{\mathcal{C}})=\widehat{\mathcal{C}}, \quad
\pi_{(1 2 3 4)}(\widehat{\mathcal{C}})=\widetilde{\mathcal{C}}.\label{eq:afro4}
\end{equation}
\end{lemma}

\proof We only prove the first of \eqref{eq:afro4}, the other relations can be proved in a similar way. Thanks to Theorem \ref{lm:coord}, a point $p\in\tilde{\mathcal C}$ parameterizes a quadruple $m$ of monodromy matrices $m:=(M_1,M_2,M_3,M_4)$ up to global diagonal conjugation. Analogously, the three projections $\widehat{q},\widecheck{q},\widebar{q}\in\widehat{\mathcal{M}}_{PVI}$ parameterize three triples of monodromy matrices $\widehat{n},\widecheck{n},\widebar{n}\in\widehat{\mathcal{M}}_{PVI}$, such that, up to global diagonal conjugation:  
\begin{align*}
&\widehat{N}_{1} = M_2,~\widehat{N}_{2} = M_3,~\widehat{N}_{3} = M_4,~\widehat{N}_{\infty} = (M_4 M_3 M_2)^{-1},\\
&\widebar{N}_{1} = M_1,~\widebar{N}_{2} = M_3,~\widebar{N}_{3} = M_4,~\widebar{N}_{\infty} = (M_4 M_3 M_1)^{-1},\\
&\widecheck{N}_{1} = M_1,~\widecheck{N}_{2} = M_2,~\widecheck{N}_{3} = M_4,~\widecheck{N}_{\infty} = (M_4 M_2 M_1)^{-1}.
\end{align*}
Now take the point $p' = \pi_{(1 2 3 4)}(p)$, this parameterizes the triple $m' = \pi_{(1 2 3 4)}(m)$ up to global diagonal conjugation.  Consider now the three projections $\widehat{q}',\widetilde{q}',\widebar{q}'\in\widehat{\mathcal{M}}_{PVI}$ of $p'$. They parameterize three triples of monodromy matrices $\widehat{n}',\widetilde{n}',\widebar{n}'\in\widehat{\mathcal{M}}_{PVI}$, such that, up to global diagonal conjugation:  
\begin{eqnarray}
&&\widehat{N}_{1}' = M_2'=M_1,~\widehat{N}_{2}' = M_3'=M_2,~\widehat{N}_{3}' = M_4'=M_3,\nonumber\\
&&\qquad\qquad\qquad\widehat{N}_{\infty}' = (M_4' M_3' M_2')^{-1}=(M_3 M_2 M_1)^{-1},\nonumber\\\
&&\widebar{N}_{1}' = M_1'=M_4,~\widebar{N}_{2}' = M_3'=M_2,~\widebar{N}_{3}' = M_4'=M_3,\nonumber\\
&&\qquad\qquad\qquad\widebar{N}_{\infty}' = (M_4' M_3' M_1')^{-1}= (M_3 M_2 M_4)^{-1},\nonumber\\\
&&\widetilde{N}_{1}' = M_1'=M_4,~\widetilde{N}_{2}' = M_2'=M_1,~\widetilde{N}_{3}' = M_3'=M_2,\nonumber\\
&&\qquad\qquad\qquad\widetilde{N}_{\infty}' = (M_3' M_2' M_1')^{-1}=(M_2 M_1 M_4)^{-1}.\nonumber\
\end{eqnarray} 
These relations show that
$$
\widehat{n}' = \widetilde{n},\quad \widebar{n}'=\pi_{(1 2 3 )}\widehat{n}, \quad \widetilde{n}'=\pi_{(1 2 3 )}\widecheck{n},
$$
where
\begin{equation*}
\pi_{(1 2 3 )}(q)= (q_3,q_1,q_2,q_\infty,q_{32},q_{21},q_{31}),\\
\end{equation*} 
Now since  $ \widetilde{n}, \pi_{(1 2 3 )}\widehat{n}, \pi_{(1 2 3 )}\widecheck{n}\in \widehat{\mathcal{M}}_{PVI}$, this shows that $p'\in \widecheck{\mathcal{C}}$.
Viceversa, we can prove in a similar way that given $p'\in \widecheck{\mathcal{C}}$, then $p=\pi_{(1 2 3 4)}^{-1} p'\in \tilde{\mathcal C}$. This concludes the proof.
\endproof

We are now ready to describe how to implement the matching algorithmically.

\subsection{Matching with the PVI 45 exceptional algebraic solutions}\label{sec:3lt}

In this Section we give an algorithm that produces the \textit{finite} set $\mathcal{C}_{\text{E}_{45}\times\text{E}_{45}\times\text{E}_{45}}$ of all \textit{candidate} points $p$ such that three over four projections $\widehat{q},\widecheck{q},\widebar{q},\widetilde{q}$, defined in \eqref{projq}, are in the set $\text{E}_{45}$.\\ 
 \begin{al}\label{al:match3lt}\quad
\begin{enumerate}
\item Consider $(\widehat{q},\widecheck{q},\widebar{q}) \in \text{E}_{45}\times \text{E}_{45} \times \text{E}_{45}$.
\item Check if $\widehat{q},\widecheck{q},\widebar{q}$ satisfy relations given by the columns of Table \ref{tb:matching}, then go to the next Step, otherwise go to Step 1.
\item Determine the two roots $p^{(i)}_{321}$, for $i=1,2$, using equation \eqref{eq:f321}.
\item[] For each $i=1,2$:
\item Calculate the values of $p_\infty^{(i)}$ using equation \eqref{eq:f1}.
\item Use Table \ref{tb:matching}  to determine all the other components of $p^{(i)}$.
\item If $p^{(i)}$ satisfies equations \eqref{eq:f6}-\eqref{eq:f11} then go to the next Step, otherwise go to  Step 1.
\item Save $p^{(i)}$ in the set  $\widetilde{\mathcal{C}}_{\text{E}_{45}\times\text{E}_{45}\times\text{E}_{45}}$, and go to Step 1.
\end{enumerate}
\end{al}

Since $\text{E}_{45}$ is a finite set, this algorithm terminates and produces a finite set $\widetilde{\mathcal{C}}_{\text{E}_{45}\times\text{E}_{45}\times\text{E}_{45}}$. Finally the big set $\mathcal{C}_{\text{E}_{45}\times\text{E}_{45}\times\text{E}_{45}}$ can be generated by Lemma \ref{lm:pperms} as follows:
$$
\mathcal{C}_{\text{E}_{45}\times\text{E}_{45}\times\text{E}_{45}} = \widetilde{\mathcal{C}}_{\text{E}_{45}\times\text{E}_{45}\times\text{E}_{45}} \bigcup_{i=1}^{3} 
					\pi_{(1 2 3 4)}^{i}(\widetilde{\mathcal{C}}_{\text{E}_{45}\times\text{E}_{45}\times\text{E}_{45}}).
$$

The Algorithm \ref{al:match3lt} together with the action of the permutations producing the set $\mathcal{C}_{\text{E}_{45}\times\text{E}_{45}\times\text{E}_{45}}$ can be found \cite{ouralgorithms}. This set contains all \textit{candidate} points $p\in \mathcal A$ such that three projections \eqref{projq} are in the set 
$\text{E}_{45}$ and consists of $3,355,200$ points.

\subsection{Matching with Okamoto-type solutions}

The set $\text{O}_{}$  of all finite orbits corresponding to algebraic solutions of Okamoto type for the PVI equation is an infinite set, therefore to construct candidate points with projections in this set is not a straightforward adaptation of Algorithm \ref{al:match3lt}. 

\begin{definition}\label{df:notrelevant}
A point $p$ is called \textit{not relevant} if the associated monodromy group is reducible or there exists an index $i=1,\dots,4,\infty$ such that $M_i = \pm \mathbb{I}$. A point $p$ is called  \textit{relevant} otherwise.
\end{definition}

In this sub-section we are going to prove a few lemmata that show that in order for $p$ to be relevant, the number  of projections corresponding to solutions of Okamoto type  is limited. We will then characterise these projections and formulate algorithms that exploit these characterisations to classify candidate points with projections of Okamoto type.

\begin{proposition}\label{thm:ok001}
If a point $p\in\mathcal A$ is such that any three of its four projections $\widehat{q},\widecheck{q},\widebar{q},\widetilde{q}$ are in the set $\text{O}_{}$  of all finite orbits corresponding to algebraic solutions of Okamoto type then the point $p$ is \textit{not relevant}.
\end{proposition}

Consequently, all points $p$ satisfying hypotheses of Proposition \ref{thm:ok001} will be irrelevant to our classification (and then excluded from it). 

Before proving this result, we will need the following two definitions:

\begin{definition}\label{df:id}
The set  $\text{O}_{\text{ID}}$ is the set of all the $q\in \text{O}_{}$ such that the associated triple of monodromy matrices $n \in \widehat{\mathcal{M}}_{PVI}$ admits one matrix equals to $\pm \mathbb{I}$.
\end{definition}

\begin{definition}\label{df:rid}
The set $\text{O}_{\text{RED}}$ is the set of all the $q\in \text{O}_{}$ such that if we consider the associated triple of monodromy matrices $n \in \widehat{\mathcal{M}}_{PVI}$ then the monodromy group $\left\langle N_1,N_2,N_3 \right\rangle$ is \textit{reducible}.
\end{definition}

{\it Proof of Proposition \ref{thm:ok001}:}
In order to prove the statement, we distinguish three cases:

\noindent (i) Assume $p$  has three projections in $\text{O}_{\text{ID}}$. It is enough to consider $m\in\hhmg2$ and the following three projections: 
\begin{equation}\label{eq-proj-prf}
\widetilde{n}=(M_1,M_2,M_3),\qquad  \widehat{n}=(M_2,M_3,M_4), \qquad
\widecheck{n}=(M_1,M_2,M_4),
\end{equation}
because all other cases differ from this case only by a permutation of the matrices $M_i$, see Lemma \ref{lm:pperms}. If any of $M_i = \pm \mathbb{I}$, then we conclude. If not, we are left with the following case:
$$
\widetilde{N}_{\infty} = M_3 M_2 M_1= \widetilde{\epsilon}\mathbb{I},\quad 
\widehat{N}_{\infty} = M_4 M_3 M_2= \widehat{\epsilon}\mathbb{I},\quad
\widecheck{N}_{\infty} = M_4 M_2 M_1= \widecheck{\epsilon}\mathbb{I},
$$
where $\widetilde\epsilon,\widehat\epsilon,\widecheck\epsilon=\pm1$. Combining these relations we obtain: 
$$
 M_4 = \widetilde{\epsilon}\widehat{\epsilon} M_1, \quad M_3 = \widetilde{\epsilon}\widecheck{\epsilon} M_4,\, \hbox{and therefore }\,
M_3 = \widehat{\epsilon}\widecheck{\epsilon} M_1, \quad  M_2 = \widetilde{\epsilon}\widehat{\epsilon} \widecheck{\epsilon}M_1^{-2},
$$ 
so that finally $m=(M_1,\widetilde{\epsilon}\widehat{\epsilon} \widecheck{\epsilon} M_1^{-2},\widehat{\epsilon}\widecheck{\epsilon} M_1,\widetilde{\epsilon}\widehat{\epsilon} M_1)$ which is reducible. Therefore $p$ is not relevant.

\noindent  (ii) Suppose $p$ is such that three projections over four are in the set $\text{O}_{\text{RED}}$. Again it is enough to consider the three projections \eqref{eq-proj-prf}.
Since the three monodromy groups defined by the triples $\widetilde{n}, \widehat{n}, \widecheck{n}$ are reducible, these triples have each a common eigenvector, let us denote them $\widetilde{v}$,  $\widehat{v}$ and $\widecheck{v}$ respectively. Now the matrix  $M_2$ that appears in all the three projections, has three eigenvectors $\widetilde{v},\widehat{v}$ and $\widecheck{v}$, which implies that one of the following identities must hold:
$\widetilde{v} = \widehat{v}\quad  \text{ or } \quad \widetilde{v} = \widecheck{v} \text{ or }\quad  \widehat{v} = \widecheck{v}.$
Therefore the  monodromy group is reducible and the point $p$ is not relevant.

\noindent  (iii) When there are three projections in $\text{O}$, not all of the same type, we apply Lemma \ref{lm:ok004}. This concludes the proof.\endproof

\begin{lemma}\label{lm:ok004}
If a point $p\in \mathcal A$ is such that one of its four projections $\widehat{q},\widecheck{q},\widebar{q},\widetilde{q}$ is in the set $\text{O}_{\text{ID}}$ and another one projection is in the set 
$\text{O}_{\text{RED}}$, then such point $p$ is \textit{not relevant}.
\end{lemma}

\proof Consider $m\in\hhmg2$ and the following two distinct \textit{generic} projections:  
\begin{align}
&(M_i,M_j,M_k)\in \text{O}_{\text{ID}},~i>j>k,~i,j,k=1,\dots,4,\label{eq:1stproj}\\
&(M_{i'},M_{j'},M_{k'})\in\text{O}_{\text{RED}},~i'>j'>k',~i',j',k'=1,\dots,4.\label{eq:2ndproj}
\end{align}
If either $M_i$, $M_j$, $M_k$ is equal to $\pm\mathbb{I}$, then we conclude. Otherwise suppose:
\begin{equation}
M_i M_j M_k = \pm \mathbb{I}. 
\label{ee1}\end{equation}
Moreover, suppose the monodromy group associated to the triple $(M_{i'},M_{j'},$ $M_{k'})$ is reducible, then the matrices $M_{i'},M_{j'},M_{k'}$ have a common eigenvector $v$. In \eqref{eq:1stproj} and in \eqref{eq:2ndproj} at least two indices $i,j,k$ that are equal to two indices $i',j',k'$, without loss of generality, suppose $i\neq i'$, $j=j'$ and $k=k'$, then equation \eqref{ee1}  implies $M_i = \pm (M_{j'} M_{k'})^{-1}$, which shows that $v$ is also an eigenvector for $M_j$ and therefore the monodromy group $<M_i,M_{i'},M_{j},M_{k}>$ is reducible as we wanted to prove.\endproof

\begin{lemma}\label{lm:ok006} Let $p$ be a relevant point such that one of its projections $q$ is in the set $\text{O}_{\text{ID}}$, then $q$ satisfies: 
\begin{equation}
q_{2 1} =  \pm q_3,\quad
q_{3 1} =  \pm q_2,\quad
q_{3 2} =  \pm q_1,\quad  q_{\infty}  = \pm 2.
\label{eq:okrecond5}\end{equation}
\end{lemma}

\proof Consider the triple of matrices $n = (N_1,N_2,N_3)$ determined by $q\in\text{O}_{\text{ID}}$. If any of the $N_i$ is equal to $\pm \mathbb{I}$, by the matching procedure, we end up with a point $p$ that is  not relevant, therefore, we avoid this case. Otherwise, assume $N_\infty = N_3 N_2 N_1 = \pm \mathbb{I}$, then:
\begin{align}
&N_1 = \pm(N_3 N_2)^{-1},~N_2 = \pm(N_1 N_3)^{-1},~N_3 = \pm(N_2 N_1)^{-1}.
\end{align}
By taking the traces we obtain \eqref{eq:okrecond5}.
This concludes the proof.
\endproof

\begin{lemma}\label{lm:ok005}
Let $q$ be the co-adjoint coordinates on $\widehat{\mathcal{M}}_{PVI}$. If $q$ is in the set $\text{O}_{\text{RED}}$, then $q$ satisfies: 
\begin{equation}
\begin{cases}
q_{i j} &=  {1 \over 2} ( q_i q_j - \epsilon_i \epsilon_j s_i s_j ),~i>j,~i,j=1,2,3,\\
q_{\infty}  &= {1\over 4} ( q_1 q_2 q_3 - \epsilon_1 \epsilon_2 s_1 s_2 q_3 - \epsilon_1 \epsilon_3 s_1 s_3 q_2 - \epsilon_2 \epsilon_3 s_2 s_3 q_1 )
\end{cases}
\label{eq:okrecond4}\end{equation}
where $s_k = \sqrt{4-q_k^2}$ for some choice of the signs $\epsilon_k = \pm 1$ for $k=1,2,3$.
\end{lemma}

\proof  Consider the triple of matrices $n = (N_1,N_2,N_3)$ determined by $q\in\text{O}_{\text{RED}}$, they define a reducible monodromy group. Therefore we can choose a basis in which they are all upper triangular. Then  their diagonal elements are given by their eigenvalues  $eigenv(N_i)=\exp{(\epsilon_l\pi\theta_i)}$, where $\epsilon_l=\pm 1$, so that:
\begin{align}
&\Tr(N_i)=2 \cos\pi\theta_i,~ i=1,2,3,\infty,\\
&\Tr(N_i N_j) = 2 \cos( \pi (\epsilon_i\theta_i + \epsilon_j\theta_j) ),~i,j=1,2,3,~i>j,\\
&\Tr(N_3 N_2 N_1) = 2 \cos( \pi (\epsilon_1\theta_1 + \epsilon_2\theta_2 + \epsilon_3\theta_3 ) ).
\end{align}
Applying trigonometric identities we obtain relations \eqref{eq:okrecond4}. This concludes the proof.
\endproof

An obvious consequence of this result is:

\begin{lemma}\label{lm:ok003}
Suppose $p\in \mathcal A$ is a relevant point such that any two of its four projections $\widehat{q},\widecheck{q},\widebar{q},\widetilde{q}$, defined in \eqref{projq}, are in the set $\text{O}_{\text{RED}}$. Denote by  $q$  one of the remaining projections, then there exists a couple of indices $(i,j)$,$(i',j')$ with one index in $(i,j)$ equal to one index in $(i',j')$ such that: 
\begin{equation}
\begin{cases}
q_{ij}^2 + q_{i}^2 + q_{j}^2 - q_{ij} q_{i} q_{j} - 4 = 0, &i>j,~i,j=1,2,3,\\
q_{i'j'}^2 + q_{i'}^2 + q_{j'}^2 - q_{i'j'} q_{i'} q_{j'} - 4 = 0, &i'>j',~i',j'=1,2,3.
\end{cases}
\label{eq:okrecond1}\end{equation}
\end{lemma}

Lemmata \ref{lm:ok006}, \ref{lm:ok005} and \ref{lm:ok003} lead to the development of additional matching algorithms in order to complete our classification for the cases when these points are included. Thanks to Lemma \ref{lm:ok004}, in order to complete our classification of candidate points, we need to construct only the following four sets: 
$\mathcal{C}_{\text{E}_{45}\times\text{O}_{\text{ID}}\times\text{O}_{\text{ID}}}$, the set of all candidate points with at least two projections in $\text{O}_{\text{ID}}$ and one in $\text{E}_{45}$, the set $\mathcal{C}_{\text{E}_{45}\times\text{O}_{\text{RED}}\times\text{O}_{\text{RED}}}$, the set of all candidate points with at least two projections in $\text{O}_{\text{RED}}$ and one in $\text{E}_{45}$, 
$\mathcal{C}_{\text{E}_{45}\times\text{E}_{45}\times\text{O}_{\text{ID}}}$, the set of all candidate points with at least two projections in $\text{E}_{45}$ and one in $\text{O}_{\text{ID}}$, and  $\mathcal{C}_{\text{E}_{45}\times\text{E}_{45}\times\text{O}_{\text{RED}}}$, the set of all candidate points with at least two projections in $\text{E}_{45}$ and one in $\text{O}_{\text{RED}}$. The set  $\mathcal{C}_{\text{E}_{45}\times\text{O}_{\text{RED}}\times\text{O}_{\text{RED}}}$ turns out to be empty.

To construct the set $\mathcal{C}_{\text{E}_{45}\times\text{O}_{\text{ID}}\times\text{O}_{\text{ID}}}$, we proceed as follows: firstly we construct the set $\widetilde{\mathcal{C}}_{\text{E}_{45}\times\text{O}_{\text{ID}}\times\text{O}_{\text{ID}}}$, where one over the three projections $\widehat{q},\widecheck{q},\widebar{q}$ is in the set $\text{E}_{45}$ and two of the remaining projections are in the set $\text{O}_{\text{ID}}$, then, applying Lemma \ref{lm:pperms} we generate the whole set $\mathcal{C}_{\text{E}_{45}\times\text{O}_{\text{ID}}\times\text{O}_{\text{ID}}}$.

The set $\widetilde{\mathcal{C}}_{\text{E}_{45}\times\text{O}_{\text{ID}}\times\text{O}_{\text{ID}}}$ is the union of the following three sets of \textit{candidate} points $p$:

\noindent (A2.1)~~$\widebar{\widetilde{\mathcal{C}}}_{\text{E}_{45}\times\text{O}_{\text{ID}}\times\text{O}_{\text{ID}}}$: \textit{candidate} points $p$ with $\widehat{q},\widecheck{q} \in \text{O}_{\text{ID}}$, $\widebar{q}\in\text{E}_{45}$.

\noindent (A2.2)~~$\widecheck{\widetilde{\mathcal{C}}}_{\text{E}_{45}\times\text{O}_{\text{ID}}\times\text{O}_{\text{ID}}}$: \textit{candidate} points $p$ with $\widehat{q},\widebar{q} \in \text{O}_{\text{ID}}$, $\widecheck{q}\in\text{E}_{45}$.

\noindent (A2.3)~~$\widehat{\widetilde{\mathcal{C}}}_{\text{E}_{45}\times\text{O}_{\text{ID}}\times\text{O}_{\text{ID}}}$: \textit{candidate} points $p$ with $\widebar{q},\widecheck{q} \in \text{O}_{\text{ID}}$, $\widehat{q}\in\text{E}_{45}$.

Here, we state only the algorithm that generates the subset (A2.1), the other algorithms for the subsets (A2.2) and (A2.3) can be derived in a similar way. 
The algorithm is based on the following result, which is an obvious consequence of Lemma \ref{lm:ok006}:
\begin{lemma}\label{lm:ok002b}
If a point $p\in \mathcal A$, is such that $\widehat{q},\widecheck{q}\in\text{O}_{\text{ID}}$, then $\widebar{q}$ must satisfy:
\begin{equation}
\widebar{q}_2 = \widehat{\epsilon}\widecheck{\epsilon} \widebar{q}_1,~\widebar{q}_{32} = \widehat{\epsilon}\widecheck{\epsilon}\widebar{q}_{31},\label{eq:id2oid1ltA}
\end{equation}
and $p$ is such that:
\begin{align}
p_1 &= \widebar{q}_1,~p_2=\widecheck{\epsilon} \widebar{q}_{31},~p_3 = \widehat{\epsilon}\widecheck{\epsilon} \widebar{q}_1,~p_4=\widebar{q}_3,~p_{21} = \widecheck{\epsilon} \widebar{q}_3,~p_{31}=\widebar{q}_{21},~p_{32} = \widehat{\epsilon} \widebar{q}_3,~\nonumber\\
&p_{41} =\widebar{q}_{31},~p_{42} = \widecheck{\epsilon} \widebar{q}_1,~p_{43} = \widehat{\epsilon}\widecheck{\epsilon} \widebar{q}_{31},~p_{432} = \widehat{\epsilon} 2,~p_{431} = \widebar{q}_\infty,~p_{421} = \widecheck{\epsilon} 2.\label{eq:id2oid1lt}
\end{align}
\end{lemma}

\begin{al}\label{al:007}\quad
\begin{enumerate}
\item Take $\widebar{q} \in \text{E}_{45}$.
\item Check if $\widebar{q}$ satisfies: 
$$
\widebar{q}_2 = \widehat{\epsilon}\widecheck{\epsilon} \widebar{q}_1,\quad\hbox{and}\quad 
\widebar{q}_{32} = \widehat{\epsilon}\widecheck{\epsilon}\widebar{q}_{31},
$$
then go to the next Step, otherwise go to Step 1.
\item Determine the components of $p$ involved in identities \eqref{eq:id2oid1lt}. 
\item Determine the values $p^{(i)}_{321}$, for $i=1,2$, using equation \eqref{eq:f321}.
\item[] For each $i=1,2$:
\item Calculate the values of $p_\infty^{(i)}$ using equation \eqref{eq:f1}.
\item Use identities given by the columns of Table \ref{tb:matching} in order to determine the other components of $p^{(i)}$.
\item If $p^{(i)}$ satisfies equations \eqref{eq:f6}-\eqref{eq:f11} then go to the next Step, otherwise Step 1.
\item Save $p^{(i)}$ in the set  $\widebar{\widetilde{\mathcal{C}}}_{\text{E}_{45}\times\text{O}_{\text{ID}}\times\text{O}_{\text{ID}}}$, and go to Step 1.
\end{enumerate}
\end{al}
When Algorithm \ref{al:007} and the algorithms for subsets (A2.2) and (A2.3) end, the following set is obtained:
\begin{equation*}
\widetilde{\mathcal{C}}_{\text{E}_{45}\times\text{O}_{\text{ID}}\times\text{O}_{\text{ID}}} = \widebar{\widetilde{\mathcal{C}}}_{\text{E}_{45}\times\text{O}_{\text{ID}}\times\text{O}_{\text{ID}}} \cup \widecheck{\widetilde{\mathcal{C}}}_{\text{E}_{45}\times\text{O}_{\text{ID}}\times\text{O}_{\text{ID}}} \cup \widehat{\widetilde{\mathcal{C}}}_{\text{E}_{45}\times\text{O}_{\text{ID}}\times\text{O}_{\text{ID}}},
\end{equation*}
then, by Lemma \ref{lm:pperms}, we generate the set $\mathcal{C}_{\text{E}_{45}\times\text{O}_{\text{ID}}\times\text{O}_{\text{ID}}}$ as:
\begin{equation}\label{eq:bigC2OID1LT}
\mathcal{C}_{\text{E}_{45}\times\text{O}_{\text{ID}}\times\text{O}_{\text{ID}}} = \widetilde{\mathcal{C}}_{\text{E}_{45}\times\text{O}_{\text{ID}}\times\text{O}_{\text{ID}}} \bigcup_{i=1}^{3} 
					\pi_{(1 2 3 4)}^{i}(\widetilde{\mathcal{C}}_{\text{E}_{45}\times\text{O}_{\text{ID}}\times\text{O}_{\text{ID}}}),
\end{equation}
where permutation $\pi_{(1 2 3 4)}$ is defined in Lemma \ref{lm:pperms}. This set contains $6,385 $ points and Algorithm \ref{al:007} can be found in \cite{ouralgorithms}.\\

We proceed in a similar way to construct the set $\mathcal{C}_{\text{E}_{45}\times\text{E}_{45}\times\text{O}_{\text{RED}}}$ of all \textit{candidate} points $p\in \hhmg2$ such that one over the four projections $\widehat{q},\widecheck{q},\widebar{q},\widetilde{q}$ is in the set $\text{O}_{\text{RED}}$ and two of the remaining projections are in the set $\text{E}_{45}$. We give here only the algorithm such that $\widehat{q},\widecheck{q} \in \text{E}_{45}$, $\widebar{q}\in\text{O}_{\text{RED}}$ - all other cases can be derived in similar way.

\begin{al}\label{al:008}\quad
\begin{enumerate}
\item Consider $\widehat{q},\widecheck{q} \in \text{E}_{45}\times \text{E}_{45}$.
\item Check if $\widehat{q},\widecheck{q}$ satisfy relations given by the columns of the first and third rows of Table \ref{tb:matching} then go to the next Step, otherwise go to Step 1.
\item Calculate $p_{31}$ and $p_{431}$ using Table \ref{tb:matching} and conditions \eqref{eq:okrecond4}.
\item Determine the values $p^{(i)}_{321}$, for $i=1,2$, using equation \eqref{eq:f321}.
\item[] For each $i=1,2$:
\item Calculate the values of $p_\infty^{(i)}$ using equation \eqref{eq:f1}.
\item Use identities given by the columns of Table \ref{tb:matching} in order to determine the other components of $p^{(i)}$.
\item If $p^{(i)}$ satisfies equations \eqref{eq:f6}-\eqref{eq:f11} then go to the next Step, otherwise Step 1.
\item Save $p^{(i)}$ in the set  $\widebar{\widetilde{\mathcal{C}}}_{\text{E}_{45}\times\text{E}_{45}\times\text{O}_{\text{RED}}}$, and go to Step 1.
\end{enumerate}
\end{al}
When Algorithm \ref{al:008} and the analogous algorithms for $\widebar{q},\widecheck{q} \in \text{E}_{45}$, $\widehat{q}\in\text{O}_{\text{RED}}$ and for
$\widebar{q},\widehat{q} \in \text{E}_{45}$, $\widecheck{q}\in\text{O}_{\text{RED}}$ respectively end, 
the  set $\widetilde{\mathcal{C}}_{\text{E}_{45}\times\text{E}_{45}\times\text{O}_{\text{RED}}}$ is obtained.
Then, as before, the set $\mathcal{C}_{\text{E}_{45}\times\text{E}_{45}\times\text{O}_{\text{RED}}}$  is given by:
$$
\mathcal{C}_{\text{E}_{45}\times\text{E}_{45}\times\text{O}_{\text{RED}}} = \widetilde{\mathcal{C}}_{\text{E}_{45}\times\text{E}_{45}\times\text{O}_{\text{RED}}} \bigcup_{i=1}^{3} 
					\pi_{(1 2 3 4)}^{i}(\widetilde{\mathcal{C}}_{\text{E}_{45}\times\text{E}_{45}\times\text{O}_{\text{RED}}}).
$$

This set contains $342, 368$ points and Algorithm \ref{al:008} can be found in \cite{ouralgorithms}.

We now produce the algorithm that generates the set $\mathcal{C}_{\text{E}_{45}\times\text{E}_{45}\times\text{O}_{\text{ID}}}$ of all \textit{candidate} points $p\in\mathcal A$ such that one projection is in the set $\text{O}_{\text{ID}}$ and two of the remaining three projections are in the set $\text{E}_{45}$. We give here only the algorithm such that $\widehat{q},\widecheck{q} \in \text{E}_{45}$, $\widebar{q}\in\text{O}_{\text{ID}}$ - all other cases can be derived in similar way. This is a simple adaptation of Algorithm \ref{al:008} in which we substitute step (2) and (8):

\begin{al}\label{al:009}\quad

(1), (2), (4), (5), (6), (7) see Algorithm \ref{al:008}.

(3) Calculate $p_{31}$ and $p_{431}$ using Table \ref{tb:matching} and conditions \eqref{eq:okrecond5}.

(8)  Save $p^{(i)}$ in the set  $\widebar{\widetilde{\mathcal{C}}}_{\text{E}_{45}\times\text{E}_{45}\times\text{O}_{\text{ID}}}$, and go to Step 1.

\end{al}

When Algorithm \ref{al:009} and the analogous algorithms for $\widebar{q},\widecheck{q} \in \text{E}_{45}$, $\widehat{q}\in\text{O}_{\text{ID}}$ and for
$\widebar{q},\widehat{q} \in \text{E}_{45}$, $\widecheck{q}\in\text{O}_{\text{ID}}$ respectively end, we obtain $\widetilde{\mathcal{C}}_{\text{E}_{45}\times\text{E}_{45}\times\text{O}_{\text{ID}}}$, then as before:
$$
\mathcal{C}_{\text{E}_{45}\times\text{E}_{45}\times\text{O}_{\text{ID}}} = \widetilde{\mathcal{C}}_{\text{E}_{45}\times\text{E}_{45}\times\text{O}_{\text{ID}}} \bigcup_{i=1}^{3} 
					\pi_{(1 2 3 4)}^{i}(\widetilde{\mathcal{C}}_{\text{E}_{45}\times\text{E}_{45}\times\text{O}_{\text{ID}}}).
$$

This set contains $ 245, 760$ points, and Algorithm \ref{al:009} can be found in \cite{ouralgorithms}.

Finally the set of all candidate points is:
\begin{equation}
\mathcal C =\mathcal{C}_{\text{E}_{45}\times \text{E}_{45}\times \text{E}_{45}}   \cup \mathcal{C}_{\text{E}_{45}\times\text{O}_{\text{ID}}\times\text{O}_{\text{ID}}} \cup\mathcal{C}_{\text{E}_{45}\times\text{E}_{45}\times\text{O}_{\text{RED}}} \cup \mathcal{C}_{\text{E}_{45}\times\text{E}_{45}\times\text{O}_{\text{ID}}}.
\end{equation}
This is a finite set consisting of $3, 461, 273$ points (duplicated points are erased). 
We re-define this set by throwing away all points that produce $M_\infty=\pm\mathbb I$, so that the resulting set $\mathcal{C}$ has $3,287,140$ elements.

\section{Extracting finite orbits}

Now, we need to determine which points in $\mathcal{C}$ lead to a finite orbit of the $P_4$-action. The following result is fundamental to achieve this:

\begin{lemma}
Let $p\in\mathcal{C}$ a \textit{candidate} point, then its orbit is finite if and only if $\beta(p)\in\mathcal{C}$ for every braid $\beta\in P_4$.
\end{lemma}

\begin{proof}
Suppose $\beta(p)\in\mathcal{C}$ for every $\beta\in P_4$, then the orbit is finite since $\mathcal{C}$ is finite too. Vice versa, suppose $p$ has a finite $P_4$-orbit, then for every $\beta$, $\beta(p)$ must have a finite orbit. Hence, $\beta(p)$ must be an element of $\mathcal{C}$. 
\end{proof}

Therefore, to select the finite orbits is equivalent to find the subset $\mathcal{C}_0\subset\mathcal{C}$ such that:
\begin{equation}
\mathcal{C}_0 = \lbrace p\in\mathcal{C}~|~\beta(p)\in\mathcal{C},~\beta\in P_4 \rbrace.
\end{equation}

To construct the set $\mathcal C_0$ we use  the following:
\begin{al}\label{al:002}\quad 
\begin{enumerate}
\item Consider $p \in \mathcal{C}$. 
\item Apply to it all the generators \eqref{mypuregens} of $P_4$.\item If there exists an $i=1,\dots,6$ such that $p^{(i)} \notin \mathcal{C}$ then delete $p$ from the set $\mathcal{C}$ and go to Step 1, otherwise save $p$ in $\mathcal{C}_0$ and go to Step 1.
\end{enumerate}
\end{al}

This algorithm ends when in the set $\mathcal{C}$ there are no more elements to delete. The final set $\mathcal C_0$ contains $1, 270, 050$ points and Algorithm \ref{al:002} can be found in \cite{ouralgorithms}.

Note that $\mathcal{C}_0$ contains only elements that generate finite orbits under the $P_4$-action. In fact, assume by contraction that $p\in \mathcal C_0$ has an infinite orbit. Then there exists a braid $\beta$ such that $\beta(p) \not\in\mathcal C$.
Now every braid $\beta\in P_4$ can be thought as an ordered 
combination of generators $\beta_{ij}$:
\begin{equation}\label{eq:lenghtword}
\beta = \underbrace{\beta_{i'j'}\dots\beta_{ij}}_{n},
\end{equation}
where $n$ indicates the length of the word. Let us introduce the following notation:
\begin{equation}\label{eq:recurs}
p^{(0)} = p,\quad p^{(1)} = \beta_{ij}(p^{(0)}), \dots,
p^{(n)} = \beta(p) = \beta_{i'j'}(p^{(n-1)}) = \underbrace{\beta_{i'j'}\dots\beta_{ij}}_{n}(p^{(0)}).
\end{equation}
Since we supposed $p^{(n)}\notin\mathcal{C}$, Algorithm \ref{al:002} deletes $p^{(n-1)}$ from the set $\mathcal{C}$. In the next iteration it deletes $p^{(n-2)}$ and so on, till when $p^{(0)} = p$ is deleted from $\mathcal{C}$, and therefore $p$ is not in $\mathcal C_0$, contradicting our hypothesis.

\section{Extracting non-equivalent orbits}

In this section we quotient the set $C_0$ of all points $p$ giving rise to a finite orbit with respect to the action of the pure braid group, so that we select only one representative point for every finite orbit, and by the action of the symmetry group $G$ of  \hmg2 described in the next theorem proved in the Appendix.

\begin{theorem}\label{thm:symm-gr} The group 
\begin{equation}
G:=\langle P_{1 3},P_{2 3},P_{3 4},P_{1 \infty},\rm{sign}_1,\dots,\rm{sign}_4,\pi_{(1 2)(34)},\pi_{(1 2 3 4)}\rangle
\label{eq:groupG}
\end{equation}
where
\begin{align}
P_{1 3}(p) =& \sigma_2 \sigma_1^{-1} \sigma_2^{-1}(p),\label{eq:P13}\\
P_{2 3}(p) =& \sigma_2 \sigma_1^{-1} \sigma_2^{-1} \sigma_1^{-1} \sigma_2 \sigma_1^{-1} \sigma_2^{-1}(p),\label{eq:P23}\\
P_{3 4}(p) =& \sigma_3\sigma_2 \sigma_1^{-1} \sigma_2^{-1} \sigma_3^{-1} \sigma_2 \sigma_1^{-1} \sigma_2^{-1} \sigma_3\sigma_2 \sigma_1^{-1}(p),\label{eq:P34}\\
P_{1 \infty}(p) =& (-p_\infty, p_2, p_3, p_4, -p_1, p_2 p_\infty - p_{432} p_{21} + p_{43} p_1 - p_{431}, \nonumber\\& p_3 p_\infty - p_{43} p_{321} + p_4 p_{21} - p_{421}, p_{32}, p_{321}, p_{42}, p_{43},\nonumber\\& p_{32} p_\infty - p_{432} p_{321} + p_4 p_1 - p_{41}, p_{432}, p_{21},\nonumber\\& p_2 p_{321} - p_{32} p_{21} + p_3 p_1 - p_{31}),\label{eq:P1infty}
\end{align}
\begin{align}
\rm{sign}_1(p)= (&-p_1,p_2,p_3,p_4,-p_\infty,-p_{21},-p_{31},p_{32},-p_{41},p_{42},p_{43},-p_{321},p_{432},\nonumber\\&-p_{431},-p_{421}),\label{eq:sign1}\\
\rm{sign}_2(p)= (&p_1,-p_2,p_3,p_4,-p_\infty,-p_{21},p_{31},-p_{32},p_{41},-p_{42},p_{43},-p_{321},-p_{432},\nonumber\\&p_{431},-p_{421}),\label{eq:sign2}\\
\rm{sign}_3(p)= (&p_1,p_2,-p_3,p_4,-p_\infty,p_{21},-p_{31},-p_{32},p_{41},p_{42},-p_{43},-p_{321},-p_{432},\nonumber\\&-p_{431},p_{421}),\label{eq:sign3}\\
\rm{sign}_4(p)= (&p_1,p_2,p_3,-p_4,-p_\infty,p_{21},p_{31},p_{32},-p_{41},-p_{42},-p_{43},p_{321},-p_{432},-\nonumber\\&-p_{431},p_{421}).\label{eq:sign4}
\end{align}
\begin{align}
\pi_{(1 2)(34)}(p) = (&p_{2}, p_{1}, p_{4}, p_{3}, p_{\infty}, p_{21}, p_{42}, p_{41}, p_{32}, p_{31}, p_{43}, p_{421}, p_{431}, p_{432}, p_{321}),\label{eq:pgen12}\\
\pi_{(1 2 3 4)}(p) = (&p_4,p_1,p_2,p_3,p_{\infty},p_{41},p_{42},p_{21},p_{43},p_{31},p_{32},p_{421},p_{321},p_{432},p_{431}).\label{eq:pgen1234}
\end{align}
is the group of symmetries for  \hmg2.
\end{theorem}

\subsection{Points belonging to the same orbit}

In this sub-section we explain how to  take the following quotient: 
\begin{equation*}
\mathcal{C}_1 := \mathcal{C}_0 / P_4.
\end{equation*}

\begin{al}\label{al:0066}\quad For every $p \in \mathcal{C}_0$:
\begin{enumerate}
\item Calculate $\mathcal{O}_{P_4}(p)$.
\item Save $p\in\mathcal{C}_1$ and delete $\mathcal{O}_{P_4}(p)$ from $\mathcal{C}_0$ .
\end{enumerate}
\end{al}
Since the set $\mathcal{C}_0$ is finite, the algorithm ends. This algorithm produces the set $\mathcal{C}_1$, that contains $17,946$ finite orbits of the $P_4$-action.

\subsection{Quotient under the symmetry group $G$}

Our aim is to quotient $\mathcal{C}_1$ by the action of the symmetry group $G$, see Appendix A. Note that $G$ is an infinite group, however it acts as a finite group on $(p_1,p_2,p_3,p_4,p_\infty)$ and preserves the length of a $P_4$-orbit. Thanks to this fact we are able to set up a finite factorization  algorithm.
 
 We proceed as follows:  we factorize by the action of the finite subgroup:
\begin{equation}\label{eq:finitesubgrouppermsigns}
\langle \text{sign}_1,\dots,\text{sign}_4,\pi_{(1 2)(34)},\pi_{(1 2 3 4)} \rangle\subset \text{G},
\end{equation}
to obtain the set $\mathcal{C}_2'$.

\begin{al}\label{al:000}\quad
\begin{enumerate}
\item Consider $p\in\mathcal{C}_1$.
\item Remove from  $\mathcal{C}_1$ the set $\mathcal{O}_{P_{4}}(p)$ and save $p$ in the set $\mathcal{C}_2'$.
\item Apply to $p$ all transformations in $\left\langle \text{sign}_1,\dots,\text{sign}_4 \right\rangle$ and save the result in the set $A_0$.
\item[] For every $p'\in A_0$:
\item Apply to $p'$ all transformations in $\left\langle \pi_{(1 2)(34)},\pi_{(1 2 3 4)} \right\rangle$ and save the result in the set $A_1$.
\item[] For every $p''\in A_1$:
\item If $p''$ is in $\mathcal{C}_1$, then $\mathcal{O}_{P_{4}}(p)$ and $\mathcal{O}_{P_{4}}(p'')$ are equivalent. Remove $\mathcal{O}_{P_{4}}(p'')$ from $\mathcal{C}_1$. If $p''$ is not in $\mathcal{C}_1$, apply again the current Step to the next $p''$ in $A_1$.
\item If all possible choices of $p''$ in $A_1$ are exhausted go to Step 1.
\end{enumerate}
\end{al}
This algorithm ends when all choices of points $p$ in the finite set $\mathcal{C}_1$ are exhausted. The set $\mathcal{C}_2'$, created in this way, contains $122$ points, therefore this factorization reduces dramatically the number of orbits to be processed  from $17,946$ to $122$.  

Next, we subdivide the set $\mathcal{C}_2'$ into subsets that contain orbits of the same length and have the same $(p_1,p_2,p_3,p_4,p_\infty)$ modulo change of signs or permutations. This is useful because, since the action of G preserves the length of an orbit and that the $(p_1,p_2,p_3,p_4,p_\infty)$ remain invariant up to permutations and sign flips under the $G$ action, only points within the same subset can be related by a transformation in $G$.

\begin{al}\label{al:factorizationONE}\quad
\begin{enumerate}
\item Consider $p\in\mathcal{C}_2'$, with $|\mathcal{O}_{P_4}(p)| = N$, $N\in\mathbb{N}$.
\item Save $p$ in a set $A_N$.
\item Remove $p$ from $\mathcal{C}_2'$. 
\item[] For every $p'\in\mathcal{C}_2'$:
\item If $p'$ is such that:
\begin{itemize}
\item $|\mathcal{O}_{P_4}(p')| = N$.
\item $(p_1,p_2,p_3,p_4,p_\infty)$  and $(p_1',p_2',p_3',p_4',p_\infty')$ differ by change of signs or permutations.
\end{itemize}
Save $p'$ in $A_N$ and remove $p'$ from $\mathcal{C}_2'$, otherwise apply again this Step to another $p'\in\mathcal{C}_2'$.
\end{enumerate}
\end{al}
Since the set $\mathcal{C}_2'$ is finite, this algorithm ends when there are no more elements in $\mathcal{C}_2'$. This algorithm generates a finite list of $54$ subsets $A_N$, where $N$ is such that for every $p\in A_N$ we have $|\mathcal{O}_{P_4}(p)| = N$ and $(p_1,p_2,p_3,p_4,p_\infty)$  and $(p_1',p_2',p_3',p_4',p_\infty')$ differ by change of signs or permutations.

Then, within each subset $A_N$, for all the elements in the subset, we apply a transformation in $G$  in such a way that every element $p$ in the subset will have the same ordered $(p_1,p_2,p_3,p_4,p_\infty)$ and check if there is a $P_4$ transformation linking the points in the same $A_N$. This is done in the following:

\begin{al}\label{al:factorizationTWO}\quad For every subset $A_N$:
\begin{enumerate}
\item Choose a point $p\in A_N$ and save it in the set $\mathcal{C}_2$.
\item Remove $p$ from $A_N$. 
\item Act with $G$ on each element in the set $A_N$, producing a new set $A_N'$ in such a way that every element $p'$ in $A_N'$ will have: 
$$
(p_1',p_2',p_3',p_4',p_\infty')=(p_1,p_2,p_3,p_4,p_\infty).
$$
\item[] For every $p'\in A_N'$:
\item Generate the orbit of $p'$ under the action of $\langle P_{13},P_{23},P_{34} \rangle$;
if $p$ is in this orbit, then $\mathcal{O}_{P_{4}}(p)$ and $\mathcal{O}_{P_{4}}(p')$ are equivalent, otherwise save $p'$ in $\mathcal{C}_2$ and apply again this Step to another $p'\in A_N'$.
\item When all choices of $p'\in A_N'$ are exhausted, go to Step 1.
\end{enumerate}
\end{al}
Since the number of subsets $A_N$ is $54$, and each subset has a finite number of elements, this algorithm ends when there are no more subsets $A_N$ to process. It turns out that  for each set $A_N$ there is only one class of equivalence under the action of the group $G$.  This completes our classification of all finite orbits. We summarize the content of the set $\mathcal{C}_2$, in Table \ref{tb:classification}.

\begin{table}[!h]
\begin{center} 
\caption {}\label{tb:classification}
{\resizebox{\textwidth}{!}{
\begin{tabular}{| c | c || c | c | c | c || c || c | c | c | c | c | c |}
\hline
\rule{0pt}{2.5ex}\rule[-1.2ex]{0pt}{0pt}\#&sz.&$p_1$ &$p_2$ &$p_3$ &$p_4$ &$p_\infty$ &$p_{21}$ &$p_{31}$ &$p_{32}$ &$p_{41}$ &$p_{42}$ &$p_{43}$ \\
\hline
\rule{0pt}{2.5ex}\rule[-1.2ex]{0pt}{0pt}$1$&$36$&$1$ &$0$ &$\sqrt{2}$ &$0$ &$0$ &$-1$ &$0$ &$-\sqrt{2}$ &$0$ &$\sqrt{2}$ &$1$ \\
\hline
\rule{0pt}{2.5ex}\rule[-1.2ex]{0pt}{0pt}$2$&$36$&$1$ &$0$ &$1$ &$0$ &$0$ &$1$ &$0$ &$1$ &$1$ &$0$ &$1$ \\
\hline
\rule{0pt}{2.5ex}\rule[-1.2ex]{0pt}{0pt}$3$&$40$&$-1$ &$1$ &$\sqrt{2}$ &$1$ &$-\sqrt{2}$ &$-1$ &$-\sqrt{2}$ &$0$ &$1$ &$1$ &$\sqrt{2}$ \\
\hline
\rule{0pt}{2.5ex}\rule[-1.2ex]{0pt}{0pt}$4$&$40$&${-1+\sqrt{5}\over 2}$ &${-1+\sqrt{5}\over 2}$ &${1+\sqrt{5}\over 2}$ &${-1+\sqrt{5}\over 2}$ &$-{1+\sqrt{5}\over 2}$ &${1-\sqrt{5}\over 2}$ &$0$ &$1$ &${1-\sqrt{5}\over 2}$ &${1-\sqrt{5}\over 2}$ &$0$ \\
\hline
\rule{0pt}{2.5ex}\rule[-1.2ex]{0pt}{0pt}$5$&$40$&$-{1+\sqrt{5}\over 2}$ &$-{1+\sqrt{5}\over 2}$ &${1+\sqrt{5}\over 2}$ &${1-\sqrt{5}\over 2}$ &${1-\sqrt{5}\over 2}$ &${-1+\sqrt{5}\over 2}$ &$-{1+\sqrt{5}\over 2}$ &$-{1+\sqrt{5}\over 2}$ &$1$ &$1$ &$-1$ \\
\hline
\rule{0pt}{2.5ex}\rule[-1.2ex]{0pt}{0pt}$6$&$45$&${-1+\sqrt{5}\over 2}$ &${-1+\sqrt{5}\over 2}$ &${-1+\sqrt{5}\over 2}$ &${1-\sqrt{5}\over 2}$ &${1+\sqrt{5}\over 2}$ &${1-\sqrt{5}\over 2}$ &${1-\sqrt{5}\over 2}$ &$-{1+\sqrt{5}\over 2}$ &$-1$ &${-1+\sqrt{5}\over 2}$ &${-1+\sqrt{5}\over 2}$ \\
\hline
\rule{0pt}{2.5ex}\rule[-1.2ex]{0pt}{0pt}$7$&$45$&${1+\sqrt{5}\over 2}$ &${1+\sqrt{5}\over 2}$ &${1+\sqrt{5}\over 2}$ &${1+\sqrt{5}\over 2}$ &${-1+\sqrt{5}\over 2}$ &${1+\sqrt{5}\over 2}$ &$1$ &${1+\sqrt{5}\over 2}$ &$1$ &$2$ &$1$ \\
\hline
\rule{0pt}{2.5ex}\rule[-1.2ex]{0pt}{0pt}$8$&$48$&$\sqrt{2}$ &$0$ &$0$ &$0$ &$\sqrt{2}$ &$\sqrt{2}$ &$-1$ &$\sqrt{2}$ &$0$ &$0$ &$1$ \\
\hline
\rule{0pt}{2.5ex}\rule[-1.2ex]{0pt}{0pt}$9$&$72$&$0$ &$0$ &$-1$ &$0$ &$0$ &$\sqrt{2}$ &$-\sqrt{2}$ &$1$ &$-1$ &$0$ &$0$ \\
\hline
\rule{0pt}{2.5ex}\rule[-1.2ex]{0pt}{0pt}$10$&$72$&$-\sqrt{2}$ &$0$ &$0$ &$-1$ &$-\sqrt{2}$ &$0$ &$-1$ &$-1$ &$\sqrt{2}$ &$-\sqrt{2}$ &$0$ \\
\hline
\rule{0pt}{2.5ex}\rule[-1.2ex]{0pt}{0pt}$11$&$81$&${-1+\sqrt{5}\over 2}$ &${1-\sqrt{5}\over 2}$ &$-1$ &${-1+\sqrt{5}\over 2}$ &${1+\sqrt{5}\over 2}$ &${-1+\sqrt{5}\over 2}$ &$1$ &$-1$ &${1-\sqrt{5}\over 2}$ &$-1$ &$0$ \\
\hline
\rule{0pt}{2.5ex}\rule[-1.2ex]{0pt}{0pt}$12$&$81$&${1+\sqrt{5}\over 2}$ &${1+\sqrt{5}\over 2}$ &$-1$ &$-{1+\sqrt{5}\over 2}$ &${-1+\sqrt{5}\over 2}$ &${1+\sqrt{5}\over 2}$ &${1-\sqrt{5}\over 2}$ &$-{1+\sqrt{5}\over 2}$ &$-1$ &$-1$ &$1$ \\
\hline
\rule{0pt}{2.5ex}\rule[-1.2ex]{0pt}{0pt}$13$&$96$&$\sqrt{2}$ &$0$ &$0$ &$0$ &$0$ &$1$ &$-\sqrt{2}$ &$\sqrt{2}$ &$1$ &$-2$ &$\sqrt{2}$ \\
\hline
\rule{0pt}{2.5ex}\rule[-1.2ex]{0pt}{0pt}$14$&$96$&${1-\sqrt{5}\over 2}$ &${1-\sqrt{5}\over 2}$ &${1-\sqrt{5}\over 2}$ &${1-\sqrt{5}\over 2}$ &${1-\sqrt{5}\over 2}$ &${1-\sqrt{5}\over 2}$ &$-{1+\sqrt{5}\over 2}$ &${1-\sqrt{5}\over 2}$ &${1-\sqrt{5}\over 2}$ &${1-\sqrt{5}\over 2}$ &${1-\sqrt{5}\over 2}$ \\
\hline
\rule{0pt}{2.5ex}\rule[-1.2ex]{0pt}{0pt}$15$&$96$&$-{1+\sqrt{5}\over 2}$ &$-{1+\sqrt{5}\over 2}$ &${1+\sqrt{5}\over 2}$ &$-{1+\sqrt{5}\over 2}$ &$-{1+\sqrt{5}\over 2}$ &${1+\sqrt{5}\over 2}$ &$-{1+\sqrt{5}\over 2}$ &$-{1+\sqrt{5}\over 2}$ &$2$ &$1$ &$-1$ \\
\hline
\rule{0pt}{2.5ex}\rule[-1.2ex]{0pt}{0pt}$16$&$96$&$0$ &$0$ &$1$ &$0$ &$-1$ &$2$ &$0$ &$0$ &$-\sqrt{2}$ &$\sqrt{2}$ &$-1$ \\
\hline
\rule{0pt}{2.5ex}\rule[-1.2ex]{0pt}{0pt}$17$&$105$&$-{1+\sqrt{5}\over 2}$ &$1$ &${1+\sqrt{5}\over 2}$ &$-1$ &${-1+\sqrt{5}\over 2}$ &$-{1+\sqrt{5}\over 2}$ &$-1$ &${1+\sqrt{5}\over 2}$ &${1+\sqrt{5}\over 2}$ &${-1+\sqrt{5}\over 2}$ &$-1$ \\
\hline
\rule{0pt}{2.5ex}\rule[-1.2ex]{0pt}{0pt}$18$&$105$&$1$ &${1-\sqrt{5}\over 2}$ &${1-\sqrt{5}\over 2}$ &$-1$ &${1+\sqrt{5}\over 2}$ &${1-\sqrt{5}\over 2}$ &${1-\sqrt{5}\over 2}$ &$-{1+\sqrt{5}\over 2}$ &$-1$ &$0$ &$0$ \\
\hline
\rule{0pt}{2.5ex}\rule[-1.2ex]{0pt}{0pt}$19$&$108$&${1+\sqrt{5}\over 2}$ &$1$ &$-{1+\sqrt{5}\over 2}$ &$-{1+\sqrt{5}\over 2}$ &${1+\sqrt{5}\over 2}$ &${1+\sqrt{5}\over 2}$ &${1-\sqrt{5}\over 2}$ &$-{1+\sqrt{5}\over 2}$ &$-2$ &$0$ &$2$ \\
\hline
\rule{0pt}{2.5ex}\rule[-1.2ex]{0pt}{0pt}$20$&$108$&${-1+\sqrt{5}\over 2}$ &${1-\sqrt{5}\over 2}$ &${-1+\sqrt{5}\over 2}$ &$1$ &${1-\sqrt{5}\over 2}$ &$-1$ &${1-\sqrt{5}\over 2}$ &$-2$ &${-1+\sqrt{5}\over 2}$ &${1-\sqrt{5}\over 2}$ &$0$ \\
\hline
\rule{0pt}{2.5ex}\rule[-1.2ex]{0pt}{0pt}$21$&$120$&$1$ &$0$ &$-1$ &$0$ &$-1$ &$0$ &$-1$ &$\sqrt{2}$ &$-\sqrt{2}$ &$-1$ &$0$ \\
\hline
\rule{0pt}{2.5ex}\rule[-1.2ex]{0pt}{0pt}$22$&$144$&${-1+\sqrt{5}\over 2}$ &$1$ &${1-\sqrt{5}\over 2}$ &${1-\sqrt{5}\over 2}$ &$-1$ &${-1+\sqrt{5}\over 2}$ &$-1$ &${1-\sqrt{5}\over 2}$ &$-1$ &${1-\sqrt{5}\over 2}$ &$-{1+\sqrt{5}\over 2}$ \\
\hline
\rule{0pt}{2.5ex}\rule[-1.2ex]{0pt}{0pt}$23$&$144$&$-1$ &$-{1+\sqrt{5}\over 2}$ &$-1$ &$-{1+\sqrt{5}\over 2}$ &${1+\sqrt{5}\over 2}$ &$1$ &${1+\sqrt{5}\over 2}$ &$1$ &$1$ &${-1+\sqrt{5}\over 2}$ &$1$ \\
\hline
\rule{0pt}{2.5ex}\rule[-1.2ex]{0pt}{0pt}$24$&$144$&$0$ &$1$ &$0$ &$0$ &$\sqrt{2}$ &$0$ &$2$ &$0$ &$1$ &$-\sqrt{2}$ &$-1$ \\
\hline
\rule{0pt}{2.5ex}\rule[-1.2ex]{0pt}{0pt}$25$&$192$&$2$ &$2$ &$-2$ &$-2$ &$-2$ &${1-\sqrt{5}\over 2}$ &$-1$ &${-1+\sqrt{5}\over 2}$ &$-{1+\sqrt{5}\over 2}$ &$-{1+\sqrt{5}\over 2}$ &$1$ \\
\hline
\rule{0pt}{2.5ex}\rule[-1.2ex]{0pt}{0pt}$26$&$192$&$0$ &$0$ &$0$ &$0$ &$0$ &$-\sqrt{2}$ &$-2$ &$-\sqrt{2}$ &$-1$ &$-\sqrt{2}$ &$-1$ \\
\hline
\rule{0pt}{2.5ex}\rule[-1.2ex]{0pt}{0pt}$27$&$200$&$0$ &$0$ &${1-\sqrt{5}\over 2}$ &${1+\sqrt{5}\over 2}$ &${1-\sqrt{5}\over 2}$ &${-1+\sqrt{5}\over 2}$ &$1$ &$1$ &${1-\sqrt{5}\over 2}$ &${1-\sqrt{5}\over 2}$ &${-1+\sqrt{5}\over 2}$ \\
\hline
\rule{0pt}{2.5ex}\rule[-1.2ex]{0pt}{0pt}$28$&$200$&${1+\sqrt{5}\over 2}$ &$0$ &$0$ &${-1+\sqrt{5}\over 2}$ &${1+\sqrt{5}\over 2}$ &${1-\sqrt{5}\over 2}$ &$0$ &${1+\sqrt{5}\over 2}$ &$1$ &${1+\sqrt{5}\over 2}$ &$-{1+\sqrt{5}\over 2}$ \\
\hline
\rule{0pt}{2.5ex}\rule[-1.2ex]{0pt}{0pt}$29$&$205$&$-1$ &$1$ &$1$ &${1+\sqrt{5}\over 2}$ &${1-\sqrt{5}\over 2}$ &$0$ &$-1$ &${1+\sqrt{5}\over 2}$ &$-{1+\sqrt{5}\over 2}$ &${1+\sqrt{5}\over 2}$ &$0$ \\
\hline
\rule{0pt}{2.5ex}\rule[-1.2ex]{0pt}{0pt}$30$&$216$&$-1$ &$0$ &$0$ &$0$ &$0$ &$0$ &$\sqrt{2}$ &$1$ &$-\sqrt{2}$ &$0$ &$1$ \\
\hline
\rule{0pt}{2.5ex}\rule[-1.2ex]{0pt}{0pt}$31$&$220$&$-1$ &$1$ &${-1+\sqrt{5}\over 2}$ &$-1$ &${-1+\sqrt{5}\over 2}$ &$-1$ &$-{1+\sqrt{5}\over 2}$ &${-1+\sqrt{5}\over 2}$ &${1-\sqrt{5}\over 2}$ &$0$ &$-{1+\sqrt{5}\over 2}$ \\
\hline
\rule{0pt}{2.5ex}\rule[-1.2ex]{0pt}{0pt}$32$&$220$&$-{1+\sqrt{5}\over 2}$ &$-1$ &$-1$ &$-{1+\sqrt{5}\over 2}$ &$1$ &${1+\sqrt{5}\over 2}$ &${-1+\sqrt{5}\over 2}$ &${1+\sqrt{5}\over 2}$ &${1+\sqrt{5}\over 2}$ &${-1+\sqrt{5}\over 2}$ &${-1+\sqrt{5}\over 2}$ \\
\hline
\rule{0pt}{2.5ex}\rule[-1.2ex]{0pt}{0pt}$33$&$240$&$1$ &$-1$ &${1-\sqrt{5}\over 2}$ &$0$ &${-1+\sqrt{5}\over 2}$ &$0$ &$-{1+\sqrt{5}\over 2}$ &${-1+\sqrt{5}\over 2}$ &$1$ &${-1+\sqrt{5}\over 2}$ &${1+\sqrt{5}\over 2}$ \\
\hline
\rule{0pt}{2.5ex}\rule[-1.2ex]{0pt}{0pt}$34$&$240$&${1-\sqrt{5}\over 2}$ &$0$ &${1-\sqrt{5}\over 2}$ &$0$ &${1-\sqrt{5}\over 2}$ &$0$ &$1$ &$0$ &$-{1+\sqrt{5}\over 2}$ &${1-\sqrt{5}\over 2}$ &$0$ \\
\hline
\rule{0pt}{2.5ex}\rule[-1.2ex]{0pt}{0pt}$35$&$240$&$1$ &$-1$ &$-{1+\sqrt{5}\over 2}$ &${1+\sqrt{5}\over 2}$ &$0$ &$-1$ &${1-\sqrt{5}\over 2}$ &$0$ &${-1+\sqrt{5}\over 2}$ &$-{1+\sqrt{5}\over 2}$ &$-{1+\sqrt{5}\over 2}$ \\
\hline
\rule{0pt}{2.5ex}\rule[-1.2ex]{0pt}{0pt}$36$&$240$&$0$ &${1+\sqrt{5}\over 2}$ &$-{1+\sqrt{5}\over 2}$ &$0$ &$-{1+\sqrt{5}\over 2}$ &$1$ &$0$ &$-1$ &${1+\sqrt{5}\over 2}$ &$-1$ &$0$ \\
\hline
\rule{0pt}{2.5ex}\rule[-1.2ex]{0pt}{0pt}$37$&$300$&${1+\sqrt{5}\over 2}$ &$1$ &$1$ &$1$ &$1$ &$1$ &$0$ &$1$ &$1$ &${1+\sqrt{5}\over 2}$ &$1$ \\
\hline
\rule{0pt}{2.5ex}\rule[-1.2ex]{0pt}{0pt}$38$&$300$&$1$ &${-1+\sqrt{5}\over 2}$ &$1$ &$1$ &$-1$ &$-1$ &${1-\sqrt{5}\over 2}$ &$-1$ &$0$ &$0$ &${1-\sqrt{5}\over 2}$ \\
\hline
\rule{0pt}{2.5ex}\rule[-1.2ex]{0pt}{0pt}$39$&$360$&$0$ &${-1+\sqrt{5}\over 2}$ &$0$ &$-1$ &${-1+\sqrt{5}\over 2}$ &$-1$ &$-1$ &$-1$ &$1$ &$1$ &$0$ \\
\hline
\rule{0pt}{2.5ex}\rule[-1.2ex]{0pt}{0pt}$40$&$360$&${1-\sqrt{5}\over 2}$ &$0$ &$0$ &$-{1+\sqrt{5}\over 2}$ &$1$ &${1+\sqrt{5}\over 2}$ &$-{1+\sqrt{5}\over 2}$ &$2$ &$1$ &$0$ &$0$ \\
\hline
\rule{0pt}{2.5ex}\rule[-1.2ex]{0pt}{0pt}$41$&$360$&$1$ &$0$ &$-{1+\sqrt{5}\over 2}$ &$0$ &$-{1+\sqrt{5}\over 2}$ &$-1$ &$0$ &$0$ &$0$ &${1+\sqrt{5}\over 2}$ &${-1+\sqrt{5}\over 2}$ \\
\hline
\rule{0pt}{2.5ex}\rule[-1.2ex]{0pt}{0pt}$42$&$480$&$1$ &$-1$ &$1$ &$1$ &$-1$ &$-1$ &$0$ &$0$ &${1-\sqrt{5}\over 2}$ &$-1$ &$1$ \\
\hline
\end{tabular}}
}
\end{center}
\end{table}
\begin{table}[!h]
\begin{center} 
{\resizebox{\textwidth}{!}{
\begin{tabular}{| c | c || c | c | c | c || c || c | c | c | c | c | c |}
\hline
\rule{0pt}{2.5ex}\rule[-1.2ex]{0pt}{0pt}$43$&$480$&$0$ &$0$ &$0$ &${-1+\sqrt{5}\over 2}$ &${-1+\sqrt{5}\over 2}$ &$-{1+\sqrt{5}\over 2}$ &$0$ &$1$ &$0$ &$1$ &$-1$ \\
\hline
\rule{0pt}{2.5ex}\rule[-1.2ex]{0pt}{0pt}$44$&$480$&$0$ &$0$ &${1+\sqrt{5}\over 2}$ &$0$ &${1+\sqrt{5}\over 2}$ &$0$ &$1$ &$0$ &${-1+\sqrt{5}\over 2}$ &$-{1+\sqrt{5}\over 2}$ &${1-\sqrt{5}\over 2}$ \\
\hline
\rule{0pt}{2.5ex}\rule[-1.2ex]{0pt}{0pt}$45$&$580$&${1-\sqrt{5}\over 2}$ &$0$ &$0$ &$0$ &${1+\sqrt{5}\over 2}$ &$0$ &${1+\sqrt{5}\over 2}$ &$-1$ &$0$ &$-2$ &$-1$ \\
\hline
\rule{0pt}{2.5ex}\rule[-1.2ex]{0pt}{0pt}$46$&$600$&$0$ &$-1$ &$0$ &${1-\sqrt{5}\over 2}$ &$-1$ &$0$ &${1-\sqrt{5}\over 2}$ &$1$ &$-{1+\sqrt{5}\over 2}$ &$0$ &$-1$ \\
\hline
\rule{0pt}{2.5ex}\rule[-1.2ex]{0pt}{0pt}$47$&$600$&$-{1+\sqrt{5}\over 2}$ &$1$ &$0$ &$0$ &$1$ &$-1$ &${-1+\sqrt{5}\over 2}$ &$-1$ &${-1+\sqrt{5}\over 2}$ &$-1$ &$-2$ \\
\hline
\rule{0pt}{2.5ex}\rule[-1.2ex]{0pt}{0pt}$48$&$900$&$0$ &$0$ &$0$ &$-1$ &${-1+\sqrt{5}\over 2}$ &$0$ &${1+\sqrt{5}\over 2}$ &${-1+\sqrt{5}\over 2}$ &${1-\sqrt{5}\over 2}$ &$-{1+\sqrt{5}\over 2}$ &$1$ \\
\hline
\rule{0pt}{2.5ex}\rule[-1.2ex]{0pt}{0pt}$49$&$900$&$0$ &$0$ &$0$ &$-1$ &$-{1+\sqrt{5}\over 2}$ &$0$ &$1$ &${-1+\sqrt{5}\over 2}$ &$-{1+\sqrt{5}\over 2}$ &${-1+\sqrt{5}\over 2}$ &${-1+\sqrt{5}\over 2}$ \\
\hline
\rule{0pt}{2.5ex}\rule[-1.2ex]{0pt}{0pt}$50$&$1200$&$0$ &$0$ &${1-\sqrt{5}\over 2}$ &$0$ &$0$ &${-1+\sqrt{5}\over 2}$ &$1$ &$1$ &$-1$ &$-1$ &$1$ \\
\hline
\rule{0pt}{2.5ex}\rule[-1.2ex]{0pt}{0pt}$51$&$1200$&$0$ &${1+\sqrt{5}\over 2}$ &$0$ &$0$ &$0$ &${1-\sqrt{5}\over 2}$ &$1$ &${-1+\sqrt{5}\over 2}$ &$1$ &$0$ &$1$ \\
\hline
\rule{0pt}{2.5ex}\rule[-1.2ex]{0pt}{0pt}$52$&$2160$&$1$ &$0$ &$0$ &$0$ &$-1$ &$0$ &$0$ &$2$ &$-1$ &$-{1+\sqrt{5}\over 2}$ &${1+\sqrt{5}\over 2}$ \\
\hline
\rule{0pt}{2.5ex}\rule[-1.2ex]{0pt}{0pt}$53$&$2160$&$0$ &$0$ &$0$ &$-1$ &$0$ &${-1+\sqrt{5}\over 2}$ &${-1+\sqrt{5}\over 2}$ &$-2$ &$0$ &$1$ &$1$ \\
\hline
\rule{0pt}{2.5ex}\rule[-1.2ex]{0pt}{0pt}$54$&$3072$&$0$ &$0$ &$0$ &$0$ &$0$ &$-{1+\sqrt{5}\over 2}$ &$0$ &$-1$ &$-{1+\sqrt{5}\over 2}$ &${1-\sqrt{5}\over 2}$ &$1$ \\
\hline
\end{tabular}}
}
\end{center}
\end{table}

\subsection{Finite monodromy groups}

Here we show that solution $25$ in Table \ref{tb:classification} corresponds to an infinite monodromy group and there is no symmetry mapping it to an orbit with finite monodromy group. We also calculate the order of the monodromy groups generated by all other orbits. The results about monodromy group orders are resumed in Table \ref{tb:classification1}.

\begin{table}
\begin{center} 
\caption {}\label{tb:classification1}
\begin{tabular}{| c   | c   |}
\hline
Orbits & Group order\\
\hline
1,3, 8, 9, 10, 13, 16, 21, 24, 26, 30  & 24\\
\hline
2 &12 \\
\hline
25& infinite\\
\hline
all others  & 60\\
\hline
\end{tabular}
\end{center}
\end{table}

To prove these statements, we calculate the monodromy matrices by using the parameterisation formulae in Section \ref{sec:coords} with the corresponding values of $p_i$, $i=1,\dots,4,\infty$ and $p_{ij}$, $i,j=1,\dots,4$ from Table \ref{tb:classification}. Since none of our groups are cyclic and they are subgroups of  $SL_2(\mathbb C)$, by Klein classification result only binary polyhedral and binary dihedral group are allowed. The order of the binary polyhedral groups is bounded by $120$, therefore we wrote a $C^{+}$ program that generates group elements up to $121$ distinct elements. In this way we characterise the orders of the monodromy groups associated to all orbits except the $25$-th one. Since for orbit $25$ all generating matrices $M_1,\dots,M_4$ are not diagonalisable and therefore not idempotent, the group is automatically infinite. This property is clearly preserved by the action of the symmetry group $G$ defined in Theorem \ref{thm:symm-gr}. For competeness we list here the monodromy matrices associated to the $25$-th orbit (in the basis of $M_3 M_2$ diagonal):

{\small{$$
M_1 = 
\left(
\begin{array}{cc}
 1-i \sqrt{\frac{2}{5+\sqrt{5}}} & \frac{1}{10} \left(5-\sqrt{5}\right) \\
 1 & 1+i \sqrt{\frac{2}{5+\sqrt{5}}} \\
\end{array}
\right)
$$
$$
M_2 = 
\left(
\begin{array}{cc}
 1-\frac{i \left(3+\sqrt{5}\right)}{\sqrt{2 \left(5+\sqrt{5}\right)}} & \frac{2+2
   \sqrt{5}-i \sqrt{2 \left(5+\sqrt{5}\right)}+i \sqrt{10 \left(5+\sqrt{5}\right)}}{4
   \left(\sqrt{5}-5\right)} \\
 \frac{ i \left(4 i \left(2+\sqrt{5}\right)+\sqrt{2
   \left(5+\sqrt{5}\right)}+\sqrt{10 \left(5+\sqrt{5}\right)}\right) }{8}& 1+\frac{i
   \left(3+\sqrt{5}\right)}{\sqrt{2 \left(5+\sqrt{5}\right)}} \\
\end{array}
\right)
$$
$$
M_3 = 
\left(
\begin{array}{cc}
 \frac{i \left(3+\sqrt{5}+i \sqrt{2 \left(5+\sqrt{5}\right)}\right)}{\sqrt{2
   \left(5+\sqrt{5}\right)}} & \frac{-1+\sqrt{5}-i \sqrt{2 \left(5+\sqrt{5}\right)}}{2
   \left(\sqrt{5}-5\right)} \\
 -\frac{ i \left(-2 i \left(1+\sqrt{5}\right)+3 \sqrt{2
   \left(5+\sqrt{5}\right)}+\sqrt{10 \left(5+\sqrt{5}\right)}\right) }{8} & -1-\frac{i
   \left(3+\sqrt{5}\right)}{\sqrt{2 \left(5+\sqrt{5}\right)}} \\
\end{array}
\right)
$$
$$
M_4 = 
\left(
\begin{array}{cc}
 -1-\frac{i \left(\sqrt{5}-3\right)}{\sqrt{2 \left(5+\sqrt{5}\right)}} &
   \frac{-1+\sqrt{5}+2 i \sqrt{2 \left(5+\sqrt{5}\right)}-i \sqrt{10
   \left(5+\sqrt{5}\right)}}{2 \left(5+\sqrt{5}\right)} \\
 \frac{ \left(-1+\sqrt{5}-2 i \sqrt{2 \left(5+\sqrt{5}\right)}+i \sqrt{10
   \left(5+\sqrt{5}\right)}\right)}{4} & \frac{i \left(-3+\sqrt{5}+i \sqrt{2
   \left(5+\sqrt{5}\right)}\right)}{\sqrt{2 \left(5+\sqrt{5}\right)}} \\
\end{array}
\right)
$$}}

\section{Outlook}
From the parameterization results of Section \ref{sec:coords}, it is clear that we could reconstruct all monodromy matrices (up to global conjugation) by matching only two points, and therefore completely reconstruct the candidate point in that way. This means that we could in fact classify all finite orbits up to two projections to Picard or Kitaev orbits. This computation is theoretically possible but extremely technical and would require covering so many sub-cases that we felt it is best to postpone it to further publications. 

Another direction of research is to classify all finite orbits of the pure braid group $P_n$ on the moduli space of $SL_2(\mathbb C)$ monodromy representations over the $n+1$-punctured Riemann sphere for $n>4$, or in other words all algebraic solutions of the Garnier system $\mathcal G_{n-2}$. We expect the matching procedure to work in this case too: now there will be $\left(\begin{array}{c}n\\3\\\end{array}\right)$ restrictions to PVI, so many more necessary conditions to be satisfied in order to produce a candidate point. We have seen that for $n=4$, we start from an extended list of $86, 768$ to produce only $54$ orbits. As $n$ increases, the starting list is the same, but the number of necessary conditions increases - therefore we expect that there will be less and less exceptional orbits as $n$ increases.






\vskip 4mm
 \noindent{\bf Acknowledgments.}
The authors are grateful to P.~A. Clarkson, G. Cousin, D. Guzzetti, O. Lisovyy, A. Nakamura and V. Rubtsov for helpful conversations. The research by PC was funded by the EPSRC DTA allocation to the Mathematical Sciences Department at Loughborough University.

\begin{appendix}\label{sec:groupG}
\section{Appendix: The symmetry group $G$ of  \hmg2}

The general theory of the bi-rational transformations of the Garnier systems was developed in \cite{kimura1990}, where Kimura proved that the symmetric group $S_5$ acts as a group of  bi-rational transformations on the Garnier system (see also 
\cite{iwasaki1991gauss}, \cite{Tsuda2006657} and \cite{JLM:253431}). These bi-rational transformations map algebraic solutions to algebraic solutions with the same number of branches. This means that the corresponding action on the co-adjoint coordinates maps finite orbits to finite orbits with the same number of points. To compute this action, we use the following result proved in 
 \cite{MazzoccoCanonical}:

\begin{lemma}
The symmetric group $S_5$ giving rise to Kimura's bi-rational transformations of the Garnier system acts on $\mathcal{M}_{\mathcal{G}_2}$ as the group $<P_{1 3},P_{2 3},$ $P_{3 4},P_{1 \infty}>$ where
the transformations $P_{1 3}$,$P_{2 3} $,$P_{3 4} $ act on the monodromy matrices as follows:
\begin{align}
P_{1 3} : (M_1,M_2,M_3,M_4) \mapsto (&M_{1}^{-1}M_2^{-1} M_3 M_2 M_1,M_2,M_2 M_1 M_2^{-1},M_4),\nonumber\\
P_{2 3} : (M_1,M_2,M_3,M_4) \mapsto (&(M_2^{-1} M_3 M_2 M_1)^{-1} M_1 M_2^{-1} M_3 M_2 M_1 ,\nonumber\\&(M_2 M_1)^{-1}M_3 M_2 M_1,M_2,M_4) ,\nonumber\\
P_{3 4} : (M_1,M_2,M_3,M_4) \mapsto (&M_{\infty}^{} M_{3}^{} M_{2}^{} M_{1}^{} (M_{\infty}^{} M_{3}^{} M_{2}^{})^{-1},M_{2}^{},\nonumber\\&(M_{3}^{} M_{2}^{} M_{1}^{}M_{2}^{-1})^{-1}M_{4}^{}(M_{3}^{} M_{2}^{} M_{1}^{}M_{2}^{-1}),M_{3}^{}),
\label{eq:trsMa}\end{align}
while transformation $P_{1 \infty}$ acts on the monodromy matrices as:
\begin{align}
P_{1 \infty} : (M_1,M_2,M_3,M_4) \mapsto &( - C_1 M_\infty C_1^{-1},C_1^{-1} M_2 C_1,C_1^{-1} M_3 C_1,\nonumber\\& C_1^{-1}  M_4 C_1),
\label{eq:trsMb}\end{align}
where $C_1$ is the diagonalizing matrix of $M_{1}$.
\end{lemma}

\begin{corollary} The group $\langle P_{1 3},P_{2 3},P_{3 4},P_{1 \infty}\rangle$ acts on the co-adjoint coordinates as in \eqref{eq:P13}, \eqref{eq:P23},\eqref{eq:P34}, \eqref{eq:P1infty}.
\end{corollary}

\proof This is a straightforward computation relying on the definition of the co-adjoint coordinates and the skein relation. \endproof

We wish to extend the class of transformations satisfying this property by adding to  $<P_{1 3},P_{2 3},P_{3 4},P_{1 \infty}>$ the following 
set of transformations that also map finite orbits to finite orbits with the same number of points (see Theorem {thm:symm-gr}):
\begin{itemize}
\item[(i)] Sign flips, or transformations that change signs to matrices $M_i$ for $i=1,\dots,4$, corresponding to  the so-called Schlesinger transformations introduced by Jimbo-Miwa in \cite{JIMBO1981407}:
\begin{align}
\text{sign}_{(\epsilon_1,\epsilon_2,\epsilon_3,\epsilon_4)} : (M_{1},M_{2},M_{3},M_{4},M_\infty) \mapsto (&\epsilon_1 M_{1}, \epsilon_2 M_{2},\epsilon_3 M_{3},\epsilon_4 M_{4},\nonumber\\&\epsilon_1\epsilon_2\epsilon_3\epsilon_4(M_4 M_3 M_2 M_1)^{-1})
\end{align}
where $\epsilon_i = \pm 1$ for $i=1,\dots,4$.
\item[(ii)] Permutations of the matrices $M_i$ for $i=1,\dots,4$ generated by:
\begin{equation}
\begin{split}\label{eq:g2perm12}
\pi_{(1 2) (34)} : (M_{1},M_{2},M_{3},M_{4},M_\infty) \mapsto ( M_{2}^{-1},M_{1}^{-1},M_{4}^{-1}, M_{3}^{-1}, M_{2}M_{1}M_{4}M_{3} ),
\end{split}
\end{equation}
\begin{equation}
\begin{split}\label{eq:g2perm1234}
\pi_{(1 2 3 4)} : (M_{1},M_{2},M_{3},M_{4},M_\infty) \mapsto ( M_{4} ,M_{1} ,M_{2} , M_{3} ,(M_{3}M_{2}M_{1}M_{4})^{-1} ).
\end{split}
\end{equation}
\end{itemize} 

The following two results give the action of the sign flips and permutations on the co-adjoint coordinates and can be proved by straightforward computations:

\begin{proposition} The sign flips are invertible maps generated by the four basic elements:
\begin{align}\label{eq:g2sign1}\rm
\rm{sign}_1: =\rm{sign}_{(-1,1,1,1)},\quad 
\rm{sign}_2 :=\rm{sign}_{(1,-1,1,1)} ,\\
\rm{sign}_3 :=\rm{sign}_{(1,1,-1,1)} \quad
\rm{sign}_4 :=\rm{sign}_{(1,1,1,-1)}\nonumber
\end{align}
that act as follow on the co-adjoint coordinates \eqref{15tuple} as in \eqref{eq:sign1}, \eqref{eq:sign2}, \eqref{eq:sign3}, \eqref{eq:sign4}.\end{proposition}

\begin{proposition} The generators $\pi_{(1 2)(34)}$ and $\pi_{(1 2 3 4)}$ act  on the co-adjoint coordinates \eqref{15tuple} as in \eqref{eq:pgen12} and \eqref{eq:pgen1234}.
\end{proposition}

Finally we characterise the group $G$ of symmetries  of  \hmg2:

\begin{definition}\label{df:symmetry}
A \textit{symmetry} for \hmg2 is an invertible map $\Phi : \hhmg2 \mapsto \hhmg2$ such that given an element $p\in\hhmg2$ and its orbit $\mathcal O(p)$, the following is true:
\begin{equation}\label{df:symm}
|{\mathcal O}(\Phi (p)) |= |{\mathcal O}(p) |.
\end{equation}
\end{definition} 

\begin{theorem}\label{thm:symm-gr} The group 
\begin{equation}
G:=\langle P_{1 3},P_{2 3},P_{3 4},P_{1 \infty},\rm{sign}_1,\dots,\rm{sign}_4,\pi_{(1 2)(34)}\pi_{(1 2 3 4)}\rangle
\label{eq:groupG}
\end{equation}
 is a group of symmetries for  \hmg2.
\end{theorem}

\proof The statement is true for the subgroup $\langle P_{1 3},P_{2 3},P_{3 4},P_{1 \infty}\rangle$ by construction. We now prove that  each generator $\Phi$ in $\langle \rm{sign}_1,\dots,\rm{sign}_4,\pi_{(1 2)(34)}\pi_{(1 2 3 4)}\rangle$ satisfies \eqref{df:symm}. It is straightforward to prove the following relations:
\begin{eqnarray}
&&
\sigma_1\rm{sign}_2 =\rm{sign}_1 \sigma_1, \quad
\sigma_1\rm{sign}_3 =\rm{sign}_3 \sigma_1, \quad
\sigma_1\rm{sign}_4 =\rm{sign}_4 \sigma_1, \nonumber\\
&&
\sigma_2\rm{sign}_1 =\rm{sign}_1 \sigma_2, \quad
\sigma_2\rm{sign}_2 =\rm{sign}_3 \sigma_2, \quad
\sigma_2\rm{sign}_3 =\rm{sign}_2 \sigma_2, \nonumber\\
&&
\sigma_2\rm{sign}_4 =\rm{sign}_4 \sigma_2, \quad
\sigma_3\rm{sign}_1 =\rm{sign}_1 \sigma_3, \quad
\sigma_3\rm{sign}_2 =\rm{sign}_2 \sigma_3, \nonumber\\
&&
\sigma_3\rm{sign}_3 =\rm{sign}_4 \sigma_3, \quad
\sigma_3\rm{sign}_4 =\rm{sign}_3 \sigma_3. \nonumber
\end{eqnarray}
so that all sign flips are indeed symmetries. Regarding the permutations, it is straightforward to prove the following relations:
\begin{eqnarray}
&&
\sigma_2 \pi_{(1 2 3 4)} = \pi_{(1 2 3 4)} \sigma_1,\quad
\sigma_1 \pi_{(1 2)(34)} = \pi_{(1 2)(34)} \sigma_1^{-1},\nonumber \\
&&
\sigma_2 \pi_{(1 2)(34)} = \pi_{(1 2)(34)} (12 34)^3 \sigma_2 \sigma_3,\quad
\sigma_3 \pi_{(1 2)(34)} = \pi_{(1 2)(34)} \sigma_3^{-1},\nonumber  \\
&&
\sigma_1 \pi_{(1 2 3 4)} = \pi_{(1 2 3 4)} \pi_{(1 2 3 4)} \sigma_2^{-1}\sigma_1^{-1},\quad
\sigma_2 \pi_{(1 2 3 4)} = \pi_{(1 2 3 4)} \sigma_1,\nonumber \\
&&
\sigma_3 \pi_{(1 2 3 4)} = \pi_{(1 2 3 4)} \sigma_2.\nonumber 
\end{eqnarray}
This conclude the proof.
\endproof

\end{appendix}


\end{document}